\documentclass[twoside,11pt]{article}
\usepackage{amsmath, amsthm, amscd, amsfonts, amssymb, graphicx, color}
\usepackage{url}
\newcommand{\weg}[1]{}
 \newcommand{\const}{\mathrm{const}}
  \newcommand{\diag}{\mathrm{diag}}
  \newcommand\bib[1]{\bibitem[#1]{#1}}
  
\newtheorem{theorem}{Theorem}[section]

\newtheorem{proposition}[theorem]{Proposition}
\newtheorem{corollary}[theorem]{Corollary}
\theoremstyle{definition}

\newtheorem{example}[theorem]{Example}

\theoremstyle{remark}
\newtheorem{remark}[theorem]{Remark}
\numberwithin{equation}{section}

\title{Projectively invariant objects and   the index of the group of affine transformations in the group of projective transformations}
 \author{Vladimir S.  Matveev}
\date{{ Mathematisches Institut,
Fakult\"at f\"ur Mathematik und Informatik,
Friedrich-Schiller-Universit\"at Jena,
07737 Jena, Germany,\\  \url{vladimir.matveev@uni-jena.de}}}

\begin{document}
\maketitle

%
\begin{abstract} 
The paper is grown from the lecture course ``Metric projective geometry'' which I hold at the summer school ``Finsler geometry with applications'' at  Karlovassi, Samos, in 2014, and at the workshop before the  8th  seminar on Geometry and Topology  of the Iranian Mathematical society at the Amirkabir University of Technology in 2015. The goal  of this lecture course was to show  how effective projectively invariant objects  can be used to solve natural and named problems in differential geometry, and  this paper also does it: I give easy new proofs to many   known statements, and also prove the following new statement: on a complete Riemannian manifold of nonconstant curvature  the index of the group of affine transformations in the group of projective transformations is at most two.   
\end{abstract} 
%
%
%
\section{Projective structure.} 
 \subsection{Definition of the projective structure} 
A slightly  informal and ineffective definition of a projective structure is as follows: projective structure on an $n$-dimensional manifold $M$ is a smooth family $\mathcal{F}$ of smooth curves such that
\begin{itemize}
\item at any point and in any direction there exists precisely one curve from this family passing through this point in this direction, and   
\item there exists an affine connection $\nabla=(\Gamma^i_{jk})$
 such that each curve from this   this family, after a proper reparameterisation, is a geodesic of this connection. 
 \end{itemize} 
 
 Note that the equation of geodesics of the connection $\Gamma^i_{jk}$
  and  of its symmetrization $\tilde \Gamma_{jk}^i= \tfrac{1}{2} \left(\Gamma^i_{jk} + \Gamma^i_{kj}\right)$ are clearly the same, since the connection comes symmetrically in the defining equation 
  \begin{equation} \label{geodesic} 
  \ddot \gamma^i + \Gamma_{jk}^i\dot \gamma^k \dot \gamma^j=0\end{equation}  
  of a geodesic; without loss of generality we will 
   therefore always assume that all connections we consider  are torsion-free. 
 
 A simplest example of a projective  structure is the family $\mathcal{F}$ consisting of all straight lines. A slightly more complicated example is when we pick 
 any connection $\Gamma^i_{jk}$ and put $\mathcal{F}$ to be  all geodesics of this connection.  Since  there is (up to a reparameterisation) an unique geodesic of a given connection passing through a given point and tangent to a given direction, the second example suggests  how one can provide a description of all projective structures: one needs to understand what connections  have the same geodesics viewed as unparameterized curves.  
 This understanding is provided by the following theorem, which was proved at least in   \cite{Levi-Civita1896}; we give the answer in the notation of  \cite{Weyl1921}.   
 
 We call   connections having 
  the same geodesics viewed as unparameterized curves \emph{projectively equivalent}.

\begin{theorem}[Levi-Civita 1896]  \label{thm:1}  $\nabla = (\Gamma_{jk}^i)$ is projectively equivalent to 
 $ \bar \nabla = (\Gamma_{jk}^i)$, if and only if there exists an $1$-form $\phi= \phi_i$ such that 
\begin{equation} 
\bar  \Gamma^i_{jk}= \Gamma_{jk}^i+ \phi_k \delta^i_j + \phi_j \delta^i_k.   \label{ast}
 \end{equation} \end{theorem}

In the index-free notation, the equation \eqref{ast} reads: 

\begin{equation} \label{ast2}
\bar \nabla_X Y=\nabla_X Y +  \phi(Y)X  + \phi(X)Y. 
 \end{equation}

We see that  the condition that two connections are projectively equivalent is quite a flexible condition:  for a given connection the set  of 
 connections that are projectively equivalent to it is an  infinitely-dimensional affine subspace in the affine space of all connections. We see also that two projectively equivalent connections can coincide at some open nonempty set  and be different on another open nonempty subset; later we will see that when we pass from connections to metrics both properties fail.

The proof of Theorem \ref{thm:1} is pretty straightforward and will be left to the reader: in the  ``$\Rightarrow$''-direction, one observes that after replacing $\Gamma $ by $\bar \Gamma$ related by \eqref{ast} in the equation \eqref{geodesic} of geodesics, the acceleration $\ddot \gamma$  of a solution 
 remains to be proportional to $\dot \gamma$ which implies that the curve $\gamma $ after a proper reparameterisation is a geodesic of $\bar \Gamma$. In the other direction one  should use the following easy fact from linear algebra: if for a 
  symmetric  w.r.t.  low indexes
  the  tensor $T^{i}_{jk}:= \Gamma^i_{jk} - \bar \Gamma^i_{jk}$ we have that for any vector $V$ the vector $T(V,V)= T^k_{ij}V^iV^j$ is proportional to $V$, then 
  $T$ has the form $T^i_{jk}= \phi_k \delta^i_j + \phi_j \delta^i_k$ for some 1-form $\phi$.

\subsection{Projective structures in dimension 2.}\label{projectiveconnection} 

Let us now consider the case of dimension $n=2$ in more details: 
If $n=2$, because of the symmetries $\Gamma_{jk}^i = \Gamma_{kj}^i$,  the components of 
 $ \Gamma_{kj}^i$ in coordinates are  given by  
$\tfrac{n^2(n+1)}{2} = 6$ functions $\Gamma^1_{11}, \Gamma^2_{11}, \Gamma^1_{12}, \Gamma^2_{12}, \Gamma^1_{22}, \Gamma^2_{22}.$ 
  The freedom in choosing the connection in the projective class is the free choice of the components   $\phi_1, \phi_2$. Thus, locally, 
 a projective structure is given by 4 functions  of the coordinates. There are many ways to encode a projective structure by 4 functions; any linear mapping  
  from $\mathbb{R}^6$ (with coordinates $\Gamma^1_{11}, \Gamma^2_{11}, \Gamma^1_{12}, \Gamma^2_{12}, \Gamma^1_{22}, \Gamma^2_{22}$) to $\mathbb{R}^4$   such that  the two-dimensional linear subspace consisting of  $\phi_k \delta^i_j + \phi_j \delta^i_k$ is its kernel, gives such a way. 
  
    Let us consider, following \cite{Beltrami1865}, one way how to encode projective structure by 4 functions; and explain its geometric sense. The following theorem is well-known:

\vspace{1ex} 

\begin{theorem}  Let  $\left[\Gamma_{jk}^i \right]$  be a projective structure on an open subset  $U\subset \mathbb{R}^2(x,y)$. 
Consider the following second order ODE  

\begin{equation}\label{eq.proj.conn.associated}
y''=\underbrace{-\Gamma^2_{11}}_{K_0} + \underbrace{(\Gamma^1_{11}-2\Gamma^2_{12})}_{K_1}y'
+\underbrace{(2\Gamma^1_{12}-\Gamma^2_{22})}_{K_2}y'^2 +\underbrace{ \Gamma^1_{22}}_{K_3} y'^3.
\end{equation}

Then, for every solution $y(x)$ of \eqref{eq.proj.conn.associated} the curve
$(x,y(x))$ is a (reparametrized) geodesic. 
\end{theorem}

\vspace{1ex} 

It is easy to check that 
the mapping 
$$( \Gamma^1_{11}, \Gamma^2_{11}, \Gamma^1_{12}, \Gamma^2_{12}, \Gamma^1_{22}, \Gamma^2_{22})\mapsto \left(-\Gamma^2_{11}, \Gamma^1_{11}-2\Gamma^2_{12}, 
2\Gamma^1_{12}-\Gamma^2_{22}, \Gamma^1_{22}\right)$$
 from $\mathbb{R}^6$  to $\mathbb{R}^4$   has two-dimensional kernel generated by the tensors  of the form $\phi_k \delta^i_j + \phi_j \delta^i_k$.

\begin{corollary} The coefficients $K_0,..., K_3$ of ODE  \eqref{eq.proj.conn.associated} contain all the information of the projective structure: two connections belong to the same projective  iff the corresponding functions $K_0,...,K_3$ coincide. \end{corollary}

\begin{example}  The   flat projective structure $[\Gamma_{jk}^i\equiv 0]$  corresponds to the ODE $y''=0$. The solutions of this ODE 
 are   $y(x)= a x + b$, and the curves $x\mapsto (x,y(x))= (x, ax + b)$ are indeed straight lines.  
\end{example} 
\vspace{2ex} 

\begin{remark}  Note that the set of curves of the form $t\mapsto (t,y(t))$ is quite big: at any point in any nonvertical 
 direction   there   exists  precisely one such  curve passing through this point in this direction. It is not a projective structure with respect to the ``informal'' definition, since there is no curve of the form $t\mapsto (t,y(t))$ tangent to the vertical direction $(0,1)$, but the difference is minor.  
\end{remark} 

\vspace{2ex}

\begin{remark}  We see a special feature of geodesics of affine connections: they are essentially the same as solutions of the  2nd order ODE $y''= F(x,y,y')$ such that $F$  is polynomial in $y'$ of degree $\le 3$.  In particular, taking  an 
 ODE $y''= F(x,y,y')$ such that $F$ is {\sc not a polynomial} in $y'$ of degree $\le 3$,    the  set of the curves  of the form $(x, y(x))$ are geodesics of no affine connection.  \end{remark} 

\section{Projectively invariant differential operators.}

\subsection{Definition  and trivial examples. }

By    \emph{projectively invariant differential operators}  we understand   differential operators on an    associated tensor bundle of $TM$
 with values in possibly another associated tensor bundle of $TM$,   constructed 
by an affine  connection on $M$ and satisfying the following condition:  if we replace the connection by another connection in the projective class the operator does not change.   

\vspace{1ex} 
  {\bf Not an  Example.} Covariant differentiation of vectors  is a differential operator  (from sections of $TM$ to sections of $T^{(1,1)}M$)  which is   not  projectively invariant:  indeed, 
if we replace $\nabla $ by a projectively equivalent $\bar \nabla$, then the covariant derivative will be changed by \eqref{ast2}.

\vspace{1ex} 
{\bf Trivial  Example.} The outer derivative $\omega \mapsto d\omega$  on the bundle  of $k$-forms with values in the bundle of $k+1$-forms  is projectively invariant. Indeed, it   does not depend on a connection at all. For example, for 1-forms, we have 
$d(adx + bdy)= \left(\tfrac{\partial b}{\partial x} - \tfrac{\partial  a}{\partial y}\right) dx\wedge dy$, and there are  no Christoffel symbols in the formula.

\vspace{1ex}
   Our next goal is to construct four   `nontrivial'  projectively invariant differential  operations, 
    two of them   will be effectively used  later.   In order to do it, we 
  need to introduce/recall   the   bundles of  weighted   tensor fields.

\subsection{ Weighted tensors.}  \label{way}

We assume that our manifold  $M$ is orientable (or we work locally) and fix an orientation. The dimension  $n$ is assumed to be $\ge 2$. We consider the bundle 
 $ \Lambda_nM $ of positive volume forms on $M$. 
 Recall that locally  a volume form is a scew-symmetic 
form of maximal order, in local coordinates $x=(x^1,...,x^n)$ 
one  can always write it as  
$f(x) dx^1 \wedge ... \wedge dx^n$ with $f\ne 0$.  The word ``positive''   means that  if the basis 
$\tfrac{\partial }{\partial x^1},..., \tfrac{\partial }{\partial x^n}$ is positively oriented, which we will always assume later,  then $f(x)>0$.

\vspace{1ex}

Positive volume forms are  naturally  organised  in    a locally trivial   1-dimensional  bundle over our manifold $M$ with the structure group $(\mathbb{R}_{>0}, \cdot)$.  
Let us discuss two natural ways for  a local trivialization of this bundle:

\begin{enumerate} \item  Choose a section in this bundle, i.e., a positive  volume form $\Omega_0=f_0  dx^1 \wedge ... \wedge dx^n$ with $f_0\ne 0$. 
Then, the other sections of this bundle can be thought to be positive functions on the manifold: the form $f   dx^1 \wedge ... \wedge dx^n$  is then essentially the same as the function $\tfrac{f}{f_0}$.  In particular,   if we change coordinates, the ratio $\tfrac{f}{f_0}$   transforms like a  function, since both coefficients, $f$ and $f_0$, are multiplied by the determinant of the Jacobi matrix.   This way to trivialize the bundle of the volume forms  will be actively used later,  and will be very effective when the volume form $f_0  dx^1 \wedge ... \wedge dx^n$ with $f_0\ne 0$ is parallel with respect to some preferred   affine connection in the projective class. Note that this trivialization is actually a global one (provided the volume form $\Omega_0$ 
 is defined on the whole manifold).

 \item In a local coordinate system  $x=(x^1,...,x^n)$, we may think that   the volume form 
 $\Omega = f(x)  dx^1\wedge ...\wedge dx^n $ corresponds to the local  function $f(x)$.   In this case, $f(x)$ can not be viewed as a function on the manifold since  
 its transformation rule is different from that of functions:   a coordinate change  
 $ x= x(y)$  transforms  $f(x)$ to $\det\left(J \right) f(x(y))$, where $  J= \left(\tfrac{dx}{dy}\right)$  is the Jacobi matrix.   
Note that this way to trivialize the bundle can be viewed as  a special case of the previous way, with the form  $ \Omega_0=   dx^1 \wedge ... \wedge dx^n$; though of course this  form $\Omega_0$ depends on the choice of local coordinates.  \end{enumerate}

   Now, take  
    $\alpha \in \mathbb{R}\setminus \{0\}$. Since $t \to t^\alpha$ is an isomorphism of $(\mathbb{R}_{>0}, \cdot)$, 
for any 1-dimensional $(\mathbb{R}_{>0}, \cdot)$-bundle  its power $\alpha$ is well-defined and is also an one-dimensional bundle. We consider 
$\left(\Lambda_n\right)^\alpha M$. It is an 1-dimensional bundle, so its sections locally can be viewed as functions. Again we have two ways to view the sections as functions:

\begin{enumerate} \item[(A)]  Choose   a volume form $\Omega_0=f_0 dx^1\wedge ... \wedge dx^n$,
 and the corresponding  section $\omega = (\Omega)^\alpha$ of
$ \left(\Lambda_n\right)^\alpha M$. Then,   the other sections of this bundle can be thought to be positive functions on the manifold.

 \item[(B)] In local coordinates $x=(x^1,...,x^n)$, we can choose the section  $(dx^1\wedge ...\wedge dx^n)^\alpha$, then the section  
 $\omega = (f(x)  dx^1\wedge ...\wedge dx^n)^\alpha $ corresponds locally to the function $(f(x))^\alpha$. Its transformation rule is different from that of functions:  a coordinate change  
 $ x= x(y)$  transforms  $(f(x))^\alpha$ to $\left(\det\left(\tfrac{dx}{dy}\right)\right)^\alpha  f(x(y))^\alpha$.   
 
 \end{enumerate}

\subsection{ Definition of weighted tensors and their covariant derivative.}  
 By a   {$(p,q)$-{\it tensor field of projective weight $k$}} we understand a section of the following bundle:

$$
T^{(p,q)}M\otimes \left(\Lambda_n\right)^{\tfrac{k}{n+1}} M    \ \ (  \textrm{notation} \  :=  T^{(p,q)}M(k))
$$

\vspace{1ex} 
  If we choose a preferred volume form on the manifold, the sections of  $T^{(p,q)}M(k)$  can be identified with 
$(p,q)$-tensor fields. The identification depends of course on the choice of the volume form.    Actually, if the chosen   volume form is parallel w.r.t. to a connection, then even the formula for the 
covariant derivative of this section coincides with that   for tensor fields.

\vspace{1ex} 
     If we do not have a preferred volume form on the manifold, in a local coordinate system one can choose  $(dx^1\wedge ...\wedge dx^n) $ as the preferred volume, and still think that sections are    ``almost'' $(p,q)$-tensors:  in a local coordinates, they are also given by $n^{p+q}$ 
     functions.  Note though that  
  their transformation rule is slightly different from that for tensors: in addition  to the usual transformation rule for tensors one needs to multiply  the result by 
$\left(\det\left(\tfrac{dx}{dy}\right)\right)^\alpha  $ with $\alpha =\tfrac{k}{n+1}$.   In particular, the formula for Lie derivative is different from that for tensors.  Also the formula for the covariant derivative is different from that for tensor: one needs to take in account the covariant derivative of $(dx^1\wedge ...\wedge dx^n) $. We are going to discuss this right now.

The  bundle of $(p,q)$-weighted tensors of weight $\alpha$  is an associated  bundle to the tangent bundle, 
so a connection $\left(\Gamma^{i}_{jk}\right)$ induces a covariant derivation on it. The next proposition shows how the covariant derivation transforms if we replace the  connection $\left(\Gamma^{i}_{jk}\right)$ by a projectively equivalent connection.

\begin{proposition}  Suppose (projectively equivalent) connections $\nabla = (\Gamma_{jk}^i)$ and $ \bar \nabla = (\Gamma_{jk}^i)$  are related by the formula  \eqref{ast2}.  
  Then, the covariant derivatives  of a volume form $\Omega\in \Gamma \left(\Lambda_n M\right) $ in the connections $\nabla $ and $\bar \nabla$ are related by   
 
\begin{equation} \label{cov1}
 \bar \nabla_X \Omega= \nabla_X \Omega - (n+1) \phi(X) \Omega. 
\end{equation}

   In particular, the covariant derivatives of the    section    $ \omega :=  \Omega^{\tfrac{k}{n+1}} \in \Gamma(\left(\Lambda_n\right)^{\tfrac{k}{n+1}} M)$
 are related by 
\begin{equation} \label{cov2} 
 \bar \nabla_X \omega = \nabla_X \omega - k  \phi(X) \omega.   
\end{equation} 
\end{proposition}

The proof of the proposition is straightforward and will be left to  the reader: the proof of \eqref{cov1} can be done by brute force calculations, and  \eqref{cov2} follows from \eqref{cov1} and from the Leibniz rule.

\subsection{A projectively-invariant differential operator on  1-forms of projective weight $(-2)$.} 

\begin{theorem}[\cite{Eastwood2008}]\label{thm:2}  For $(0,1)$-tensors of projective weight (-2) the differential operator  
\begin{equation}\label{K1} 
T \mapsto \operatorname{Symmetrization\_Of} (\nabla T)   \end{equation} 
 is projectively invariant: it does not depend on the choice of the affine connection in the projective class. \end{theorem} 
 
 \vspace{1ex} {\bf Proof.}  
Let $K\in \Gamma\left(T^{(0,1)} M (-2)\right)$ be an 1-form of projective weight $(-2)$.
 We calculate the difference of  their   $\nabla$- and $\bar \nabla$- derivatives  assuming  \eqref{ast2}:
\begin{equation} \label{3} 
\bar \nabla_X K =  \nabla_X K \underbrace{- \phi(X) K - K(X) \phi}_{\textrm{\tiny because of \eqref{ast2}}} + \underbrace{ 2 \phi(X) K}_{\textrm{\tiny because of \eqref{cov2}}} = \nabla_X K  +  \phi(X) K - K(X) \phi.  
\end{equation}

We see that  $(\bar \nabla_X K) (Y)-(  \nabla_X K) (Y)$  is scewsymmetric in $X,Y$; then    it vanishes after symmetrization.  \qed 
 
\begin{remark}  In the index notation, the mapping \eqref{K1}  reads $K_i\mapsto K_{i, j} + K_{j, i}$, where ``comma'' denotes the covariant differentiation, as described above.   The equation $ K_{i, j} + K_{j, i}=0$ is called    the \emph{projective Killing equation} for weighted 1-forms.  In coordinates, the formula for the covariant derivative depends on what way, (A) or (B) from \S \ref{way}, we have chosen for representing the form in local coordinates. If we have chosen the way (A) and if in addition the volume form $\Omega_0=f_0 dx^1\wedge ... \wedge dx^n$ is parallel with respect to the connection $\Gamma$ we will use for covariant differentiation, then the formula for  the covariant derivative is precisely the same as the one for usual  tensor fields:
$$
K_{i, j}= \frac{\partial K_i}{\partial x^j} - K_s  \Gamma^s_{ij}.
$$
Now, if we have chosen the way (B), then we obtain the formula 
$$
K_{i,j} =  \frac{\partial K_i}{\partial x^j} - K_s  \Gamma^s_{ij} +\tfrac{2}{n+1}{K_i} \Gamma^s_{js}.
$$

 \end{remark}

\begin{corollary}  \label{cor:2}   For  symmetric  $(0,2)$-tensors of projective weight $(-4)$ the operation 
$$
K \mapsto \operatorname{Symmetrization\_Of} (\nabla K)
$$
 is projectively invariant: it does not depend on the choice of the affine connection in the projective class. 
 \end{corollary}

 {\bf    Proof.} Decompose symmetric  $(0,2)$ tensors of weight $(-4)$ into 
   the  sum of symmetric tensor 
   products  of $(0,1)$ tensors of weight $(-2)$ and apply Theorem \ref{thm:2}.  \qed
 
  The equation $\operatorname{Symmetrization\_Of} (\nabla K)=0$ (on symmetric  $(0,2)$-tensors of projective weight $(-4)$)   is called   the \emph{projective   Killing equation}; it  will play an important role  later. In the index form it reads 
\begin{equation} \label{K2} 
  K_{ij,k} +   K_{jk,i} +   K_{ki,j}=0.
 \end{equation}

Let us introduce two more projectively invariant operators: 

\begin{theorem} \label{thm:3}   For $(1,0)$-tensors of projective weight 1 the operation 
$$
v \mapsto Trace\_Free\_Part\_Of \nabla(v)  = v^i_{ \  , j} - \tfrac{1}{n} v^s_{\ ,s} \delta^i_j.   
$$
 is projectively invariant.  For symmetric $(2,0)$-tensors   $\sigma$  of projective weight $2$ the operation 
\begin{equation} \label{E4}
\sigma^{ij} \mapsto \sigma^{ij}_{ \ \ , k} - \tfrac{1}{n+1}   ( \sigma^{is }_{ \ \ , s} \delta^j_k +  \sigma^{js }_{ \ \ , s} \delta^i_k)  
\end{equation} 
 is projectively invariant.  \end{theorem} 

\vspace{1ex} 

The proof will be left to a reader: the proof of the first statement is similar to that of Theorem \ref{thm:2}, and the proof of the second statement is similar to that of Corollary \ref{cor:2}. 

\vspace{1ex}

   \begin{remark} In the index-free notation the operation \eqref{E4}  reads 
 $$
 \sigma \mapsto \operatorname{Trace\_Free\_Part\_Of} \left(\nabla \sigma\right).   
 $$
 Though $\nabla \sigma$ is a (2,1)-weighted-tensor,  its  trace is well-defined and  is a (1,0)-tensor of projective weight 2 given by the formula $\operatorname{trace}(\sigma^{ij}_{ \ \  ,k})= \sigma^{sj}_{ \  \ , s} .$
 \end{remark}

\subsection{ Geometric importance of the operator $
 \sigma \mapsto \operatorname{Trace\_Free\_Part\_Of} \left(\nabla \sigma\right)$}

  \begin{theorem} \label{thm:4} Suppose the Levi-Civita connection  of a metric $g$ lies in a projective class $\left[ \nabla \right]$. Then,  $\sigma^{ij} := g^{ij} \otimes \left(  \operatorname{Vol}_g\right)^{\tfrac{2}{n+1}}$ is a solution of  
\begin{equation} 
 \operatorname{Trace\_Free\_Part\_Of} \left(\nabla \sigma\right)=0. \label{5}     
 \end{equation}
  Moreover, for every  solution of the equation \eqref{5}  such that $\det(\sigma) \ne 0$ there exists a metric $g$  whose Levi-Civita connection lies in the projective class and such that $\sigma^{ij} := g^{ij} \otimes \left(  \operatorname{Vol}_g\right)^{\tfrac{2}{n+1}}$. 
\end{theorem} 

Theorem \ref{thm:4} is due to \cite{EastwoodMatveev2007}. Its two-dimensional version was essentially known to \cite{Liouville1889}.

\vspace{1ex} 
  \noindent {\bf Proof in  the direction $\Rightarrow$. }   We assume that $\nabla^g\in \left[ \nabla \right]$. Since our equation is projectively invariant, we may choose  any connection in the projective class; w.l.o.g. we choose   the Levi-Civita connection $\nabla^g$. In this connection the metric and therefore 
all objects constructed by the metric are parallel so $\nabla^g(\sigma)= 0$ which of course implies \eqref{5}.   \qed

\vspace{1ex} 
  \noindent {\bf Proof in  the direction $\Leftarrow$. } 
Let us  observe that though the operator    $$
\sigma^{ij} \mapsto \sigma^{ij}_{ \ \ , k} - \tfrac{1}{n+1}   \left( { \sigma^{is }_{ \ \ , s} \delta^j_k +  \sigma^{js }_{ \ \ , s} \delta^i_k}\right)$$
does not depend on the choice of a connection in the projective class, the   terms  in the right hand side  do depend. Indeed, by direct calculations we see that 
$$
\bar \nabla_k \sigma^{ij} -   \nabla_k \sigma^{ij} = \sigma^{is} \phi_s \delta^j_k+ \sigma^{js} \phi_s \delta^i_k \ .
$$
If $\sigma^{ij}_{ \ \ , k} - \tfrac{1}{n+1}   \left( { \sigma^{is }_{ \ \ , s} \delta^j_k +  \sigma^{js }_{ \ \ , s} \delta^i_k}\right) =0$, 
this implies   
$$
\bar \nabla_k \sigma^{ij} =  \tfrac{1}{n+1}   \left( { \sigma^{is }_{ \ \ , s} \delta^j_k +  \sigma^{js }_{ \ \ , s} \delta^i_k}\right)  +  \sigma^{is} \phi_s \delta^j_k+ \sigma^{js} \phi_s \delta^i_k.  
$$

 Thus, if we as  the 1-form $\phi$  take the one satisfying the condition 
 $(n+1) \sigma^{is} \phi_s = -\sigma^{is }_{ \ \ , s} $, which is always possible if  $\sigma^{is }$ is nondegenerate, we obtain that $\sigma$ is paralell with respect to $\bar \Gamma$ implying that $\bar \Gamma$ is the Levi-Civita connection of the corresponding metric. \qed

  Let us now explain the relation between (nondegenerate) solutions $\sigma$ 
 of  the metrisability equation  \eqref{5} and metrics in coordinates.  
 
 Let us work in a coordinate system and choose $dx^1\wedge ... \wedge dx^n$
  as a  volume form, i.e., we have chosen the way (B) from \S \ref{way} to do local calculations.

 \begin{itemize} 
 
\item  If we have a metric $g_{ij}$, then the corresponding solution of the metrisability equation is given by  the matrix 
\begin{equation} \label{eq:sigma1}
 \sigma^{ij} :=  \left(g^{ij} \otimes \left(  \operatorname{Vol}_g\right)^{\tfrac{2}{n+1}}\right)= g^{ij} |\det g|^{\tfrac{1}{n+1}}.
\end{equation}

\item For  a solution  $\sigma= \sigma^{ij}$ of the metrisability  equation such that its determinant in not zero, the corresponding metric is given by 
\begin{equation} \label{eq:sigma}
 g^{ij} := |\det(\sigma)| \sigma^{ij}. 
\end{equation}

\end{itemize}

\begin{remark} \label{formula:sinjukov} 
In there exists a metric in the projective class, one can use its Levi-Civita connection for covariant differentiation and its volume form for identifications of weighted tensors with tensors. After doing this and using that the volume form is parallel, the formula \eqref{5} reads 
\begin{equation} \label{eqn:sinjukov}
a^{ij}_{ \ \ , k}= \lambda^i\delta_k^j+ \lambda^j\delta_k^i,  
\end{equation} 
where $a^{ij}$ now is a (symmetric) (2,0)-tensor related to $\sigma^{ij}$ from  \eqref{5} by $a = \sigma \otimes \left(  \operatorname{Vol}_g\right)^{\frac{2}{n+1}}$. 
This formula was known before, see e.g. \cite{Sinjukov1979} or \cite{BolsinovMatveev2003}. 

Note that contracting \eqref{eqn:sinjukov} with $g_{ij}$ we see that the vector field 
$\lambda^i$ is actually the  half of the $g$-gradient of the $g$-trace of $a$,
$$
\lambda^i= \left(\tfrac{1}{2} g^{is} \left(a^{pq }g_{pq}\right)_{, s}\right).   
$$
  In particular, if all eigenvalues of $A^i_j:= a^{pi }g_{pj}$ are constant, $\lambda^i$ is zero and therefore
  $\sigma$  is parallel. In particular,  if $\sigma$ came from a projectively equivalent metric, then this metric is actually affinely equivalent to $g$. 
\end{remark} 

As an example let us consider  the case  of dimension   two.  
As we explained in \S \ref{projectiveconnection}, in dimension 2 the  four functions   $K_0, K_1, K_2, K_3$ (coefficients  of the  ODE \eqref{eq.proj.conn.associated})   determine the projective class.

In this setting, the metrisability  equations in the following system of 4 PDE on three unknown functions:

\begin{equation} \label{metrization} \left\{\begin{array}{rcc}
{\sigma^{22}}_x-\tfrac{2}{3}\,K_1\,\sigma^{22} -2\,K_0\,\sigma^{12}&=&0\\
{\sigma^{22}}_y-2\,{\sigma^{12}}_x
-\tfrac{4}{3}\,K_2\,\sigma^{22}-\tfrac{2}{3}\,K_1\,\sigma^{12}+2\,K_0\,\sigma^{11}&=&0\\
-2\,{\sigma^{12}}_y+{\sigma^{11}}_x
-2\,K_3\,\sigma^{22}+\tfrac{2}{3}\,K_2\,\sigma^{12}+\tfrac{4}{3}\,K_1\,\sigma^{11}&=&0\\
{\sigma^{11}}_y+2\,K_3\,\sigma^{12}+\tfrac{2}{3}\,K_2\,\sigma^{11} &=&0
\end{array}\right.
\end{equation}

 In higher dimensions, the metrisability  equations in also an overdetermined system of    PDE. In dimension $n$, it has $\tfrac{n(n+1)}{2}$ unknowns and $\tfrac{n^2(n+1)}{2}-n$ equations; the coefficients are constructed by certain explicit formulas by coefficients of a  connection and do not depend on the choice of connection within the projective class. 

\begin{corollary} \label{duna} Generic (in the $C^\infty$-topology)  projective structure is not metrizable (assuming $n=\dim M\ge 2$). 
\end{corollary} 

{\bf Explanation.} It is known  that the existence of an  nontrivial solution of an overdetermined system implies that the coefficients of this system satisfy certain algebraic-differential relations (known as ``integrability conditions''). In our case, one can show that the integrability conditions do not vanish identically  and therefore are not zero for  a generic metric, which implies  Corollary \ref{duna}. The   proof  that the integrability conditions do not vanish identically requires some work; for dimension 2 it was done in   \cite{BryantDunajskiEastwood2009}.  
In dimension 3, instead of \cite{BryantDunajskiEastwood2009} on can use   \cite{DunajskiEastwood2014}. In other dimensions   one needs to slightly and straightforwardly generalize certain results of  \cite{DunajskiEastwood2014}.

  \section{ Metric Projective Geometry}  
  
  \subsection{Philosophy and goals.  }

One can of course study  projective structures without thinking about whether there is a (Levi-Civita connection of a)  metric in the projective class.  
Unfortunately, in this case there are only few  ``easy to formulate, hard to prove'' results, and we are not aware of any  applications to  
 or interplay  with  other branches of mathematics and other sciences. Of course, there are plenty of applications of say representation theory to the theory of projective structures (see e.g. \cite{CapSlovak2009}), but not in the other directions.  
 
We suggest to  study metrizable projective structures, i.e.,  such that there exists a metric in the projective class.    The condition that the projective structure is metrizable is a strong condition: as we mentioned above,  generic projective structure is not metrizable. Moreover, 
generic   metrizable projective structure has only one, up to a scaling, metric in the projective class (see e.g. \cite{Matveev2012a}). In this case, all geometric questions can be reformulated as questions to this metric (say, projective vector fields for such projective structure are automatically   homothety vector fields for this metric).

So, in what follows  we will concentrate on metrizable projective structures such that  \emph{there exists at least two nonproportional metrics in the projective class}.  
It is of course a natural object of study; there are a lot of results in this topic going back to \cite{Lagrange1789}, \cite{Beltrami1865}, \cite{Dini1869}, \cite{Levi-Civita1896} and so on; we recall  and reprove some of them.  We will see that in this topic there are 
many  ``easy to formulate, hard to prove'' results,   many named  and natural problems, and there is a deep interplay with other branches of mathematics (in our paper we will use a relation to the theory of integrable systems); see also    e.g. \cite{Matveev2012a} for explaining how this topic  appeared within general relativity. In the next section we will defend this viewpoint by  proving such  ``easy to formulate, hard to prove'' result; the proof will actually be relatively easy using the projective invariant equations   we explained before.

\subsection{Topology of 2-dimensional manifolds admitting projectively equivalent metrics.  }
Our goal will be to prove the following theorem which was first proved in \cite{MatveevTopalov1998}.  The present proof is a new one. 

\begin{theorem} \label{thm:5}  Let $(M^2, g)$ be a two-dimensional closed (compact, no boundary) Riemannian  manifold.  Assume a metric $\bar g$ is projectively equivalent to $g$ and is nonproportional to $g$. Then, $M^2$ has nonnegative Euler characteristic. 
\end{theorem}

In other words,     surfaces  of genus $\ge 2$  do  not admit nonproportional projectively equivalent metrics. 

Note that surfaces of genus $0 $ and $1$ do admit nonproportional projectively equivalent metrics:  the existence of such metrics on the 2-torus (and also on the Klein bottle)  follows immediately from Theorem \ref{thm:dini},  and the existence of such metrics on the sphere  (and also on the projective plane)  follows from  Example \ref{BE}.

In order to prove Theorem \ref{thm:5}, we need to do some preliminary work; though for us the most interesting is the dimension $2$, this preliminary work  is  valid  in any dimension $n\ge 2$, and will be also used later in all dimensions.   The proof of Theorem \ref{thm:5} starts in \S \ref{proof:5}.

\begin{proposition} \label{prop:int1}  Let a projective structure $[ \Gamma]$ contains the  Levi-Civita connection of a metric $g$. Then, the weighted $(0,2)$-tensor 
$K_{ij}= g_{ij} \otimes (  \operatorname{Vol}_g)^{\tfrac{-4}{n+1}}$ of projective weight $(-4)$ is a solution of the projective Killing equation \eqref{K2}.   
\end{proposition}

{\bf Proof.}  Projectively invariant equations do not depend on the choice of connection in the projective class, w.l.o.g. we can therefore take the Levi-Civita connection of the metric $g$. Then, the covariant derivatives of  $g$ and of $  \operatorname{Vol}_g$ are zero implying that  the covariant derivative of 
$K$ vanishes even without symmetrization. \qed

\begin{proposition} \label{prop:int2} Suppose a weighted (0,2) tensor $K$ is a solutions of the projective Killing equation. Then, for any metric $g$ in the projective class 
 the \underline{(unweighted)} tensor field
$$
\hat K:= K\otimes  (  \operatorname{Vol}_g)^{\tfrac{4}{n+1}}.
$$
is a Killing tensor, that is  it satisfies the Killing equation  $$\operatorname{Symmetrization\_Of} \nabla \hat K= \hat K_{ij,k} + \hat K_{jk,i}+ \hat K_{ki, j}=0.$$  
\end{proposition}

{\bf    Proof.}  $(  \operatorname{Vol}_g)$, and therefore, $(  \operatorname{Vol}_g)^{\tfrac{4}{n+1}}$ is parallel w.r.t. the Levi-Civita connection of $g$. Then, 
$\nabla(K\otimes  (  \operatorname{Vol}_g)^{\tfrac{4}{n+1}})= (\nabla K)\otimes  (  \operatorname{Vol}_g)^{\tfrac{4}{n+1}}, $ and therefore 
 $$\operatorname{Symmetrization\_Of} \left(\nabla\left(K\otimes  (  \operatorname{Vol}_g)^{\tfrac{4}{n+1}} \right) \right)= \left(\operatorname{Symmetrization\_Of} (\nabla K)\right)\otimes  (  \operatorname{Vol}_g)^{\tfrac{4}{n+1}}=0$$ implying the claim.
 \qed

 Recall now the geometric sense of the Killing tensors:    a (0,2) tensor field  $K= K_{ij} $ is a Killing tensor for a metric, if and only if the function 
 $$
 I_K:TM\to \mathbb{R}, \ \  I_K(\xi)= K_{ij}\xi^i\xi^j
 $$ is an integral of the geodesic flow of the metric $g$, i.e., for any arclenght  parameterized geodesic $\gamma$ we have that the function 
 $t\mapsto I_K(\dot \gamma)$ is constant (of course it may depend on geodesic  but for a fixed geodesic 
 does not depend on $t$).   Indeed,  the $\frac{d}{d t}-$derivative of the  function $t\mapsto I_K(\dot \gamma)$ is equal to  \begin{equation}  \label{star} 
\nabla_{{   \dot \gamma}}\left(K(\dot \gamma, \dot \gamma)\right).  
\end{equation} 
  By of the definition of geodesics   $\nabla_{{   \dot \gamma}} \dot \gamma=0$, so  \eqref{star} reduces to 
$$
\nabla K({   \dot \gamma}, \dot \gamma, \dot \gamma)=0,    
$$  which is equivalent to \    $\operatorname{Symmetrization\_Of} (\nabla K)=0$.

\begin{example}[Trivial integral: energy]   If we first use Proposition \ref{prop:int1} to  construct a projective Killing tensor by a metric $g$, and then use  this  projective Killing tensor  to construct a Killing tensor by Proposition \ref{prop:int2}, we obtain $\hat K=g$ which is of course a Killing tensor; the corresponding integral is (up to a coefficient $2$)  the kinetic energy. 
 \end{example}

A nontrivial Killing tensor  appears, if we have two nonproportional metric in the projective class: suppose the metric $\bar g$ is projectively equivalent to $g$. Then, applying  Proposition \ref{prop:int1} for the metric $\bar g$, we obtain that $\bar g \otimes (  \operatorname{Vol}_{\bar g})^{\tfrac{-4}{n+1}}$ is a projective Killing tensor. Then, applying  Proposition \ref{prop:int2} we obtain that $\bar g \otimes (  \operatorname{Vol}_{\bar g})^{\tfrac{-4}{n+1} }\otimes (  \operatorname{Vol}_{ g})^{\tfrac{4}{n+1}} $ is a Killing tensor. Note that $(  \operatorname{Vol}_{\bar g})^{\tfrac{-4}{n+1}} \otimes (  \operatorname{Vol}_{ g})^{\tfrac{4}{n+1}}$ is  actually a function given by $\left| \frac{ \det g }{\det \bar g} \right|^{\frac{2}{n+1}}.$ We just have proved the following theorem:

 \begin{theorem}  \label{thm:int} Let $g$ and $\bar g$ be  projectively equivalent. Then, the function  
  \begin{equation} \label{int} 
 I(\xi) = \left|\tfrac{\det g }{\det \bar  g  }\right|^{\tfrac{2}{n+1}}
 \bar g(\xi, \xi) 
 \end{equation} 
 is an integral for the geodesic flow of $g$.
 \end{theorem}

{\bf    Historical remark.} We do not pretend that Theorem \ref{thm:int}  is new, or that the proof we give is most effective. 
There are more direct proofs  that $I$ is an   integral, and the statement itself was known at the end of the 19th century, see e.g. \cite{Painleve1897}. The importance of this statement was not fully understood though until it was rediscovered in \cite{MatveevTopalov1998}; we show how effective Theorem \ref{thm:int} can work in the proof of Theorem \ref{thm:5} .

\subsection{ Proof of Theorem \ref{thm:5}} \label{proof:5}

 In dimension 2, the integral \eqref{int} reads     
\begin{equation} \label{int:dim2}  I(\xi):= \left|\frac{\det( g)}{\det(\bar g)}\right|^{\frac{2}{3}}\bar g(\xi,\xi).\end{equation}

  Assume our closed  surface $M^2$ has negative Euler characteristic (w.l.o.g. we assume that the surface is oriented; then it   has genus $\ge 2$).   The goal is to show that  projectively equivalent  $g$ and $ \bar g$ are proportional. 

\vspace{1ex} 
  Because of topology, there exists $p$ such that   
 $g_{|p}=\const \cdot \bar g_{|p}$. Indeed, otherwise the eigendirections of the (1,1)-tensor field $g^{is} \bar g_{sj}$ will give, at least on the 2-cover, two 
 1-dimensional distributions, which is possible only on surfaces of zero Euler characteristic.  
 
 \vspace{1ex} 
   W.l.o.g. we assume $\const = 1$; we can do it since after 
 multiplying the metric $\bar g$ by a  nonzero  constant we obtain a projectively equivalent metric. 
 We assume  that at a point $q$ we have    $g_{|q}\ne \bar g_{|q}$ 
and find a contradiction.

First observe that, because of the metrics do not coincide  at $q$, 
 the set $$A:= \{\xi\in T_{q}M \mid I(\xi)= 1, g(\xi,\xi)=1\}$$   is the intersection of two different quadrics and  contains at most 4 points. Now, for any arc-length parameterized 
 geodesic $\gamma$ 
 connecting $p$ with $q$ (we assume $\gamma(0)= q $ and $\gamma(L)= p$,  where $L$ is the length of geodesic)  we have that $\dot\gamma(0)\in A$. Indeed, 
   $g( \dot\gamma(0), \dot\gamma(0))= 1$, since the geodesic is arc-length parameterized, and $I(\dot\gamma(0))=1$, since $I$ is an integral so $I(\dot\gamma(0))= I(\dot\gamma(L))$, and at the point $\gamma(L)= p$ the metrics coincide  so $I(\xi)= g(\xi, \xi)$ by \eqref{int:dim2}.

But because of the topology there are a lot of geodesics connecting $p$ and $q$. In fact,  from  the Hopf-Rinow Theorem  in follows that the number $N_R$  of geodesics of length $\le R$ connecting $q$ and $p$ grows  exponentially in  $R$.  From the other side, the initial velocity vectors of all such geodesics lie  in the finite set $A$, so the number $N_R$ can not grow faster than linearly in $R$. This gives us a contradiction, which proves Theorem \ref{thm:5}. \qed

\begin{remark} \label{rem:conformal} We also see that   two projectively equivalent 
 metrics  can not be proportional   with different coefficients of the proportionality at  points that can be connected by a geodesic, because in this case the set $A$ or its analog for higher dimensions is simply empty. In particular, two conformally equivalent metrics can not be projectively equivalent unless the conformal coefficient is constant (the latter  result is known and for dimensions $\ge 3$ is due to \cite{Weyl1921}). 
\end{remark}

\subsection{ Local normal forms of projectively equivalent  2-dimensional Riemannian metrics.}

 The following question has been explicitly asked in \cite{Beltrami1865}:

 \vspace{1ex} 
\noindent{\bf Local normal form question:} Given two projectively equivalent metric, how do they look in ``the best'' coordinate system (near a generic point)? How unique is such a  coordinate system?

\vspace{1ex} 

Answer in dimension 2 was obtained by Dini; our next goal is to reprove the Dini's theorem below. 

\vspace{1ex} 

\begin{theorem}[\cite{Dini1869}] \label{thm:dini} Let $g$ and $\bar g$ be projectively 
equivalent  twodimensional  Riemannian metrics. Then, in a neighborhood of almost every point there exists a coordinate system such that in this coordinate system the metrics are 
\begin{eqnarray}  \label{dini1}  
g &=& \begin{pmatrix} X(x) - Y(y) & \\  & X(x) - Y(y) \end{pmatrix} \\
\bar g & = &    \begin{pmatrix} \tfrac{X(x) - Y(y)}{X(x)^2Y(y)}   & \\  & \tfrac{X(x) - Y(y)}{X(x)Y(y)^2} \end{pmatrix}= \left(\frac{1}{Y(y)}-\frac{1}{X(x)} \right)\begin{pmatrix} \tfrac{1}{X(x)}   & \\  & \tfrac{1}{Y(y)} \end{pmatrix}, \label{dini2}
\end{eqnarray}
where $X(x)$ and $Y(y)$ are functions of the  indicated variables. 
 The coordinates are unique modulo  
 $(x,y)\mapsto (\pm x + b,  \pm y + d)$.

 Moreover, for any functions $X(x)$ and $Y(y)$ such that the matrices  (\ref{dini1}, \ref{dini2})  are nondegenerate, the  metrics (\ref{dini1}, \ref{dini2}) are projectively equivalent. 
 \end{theorem} 
 
\begin{remark} Actually,  the answer to the question of Beltrami is known  in all dimensions and in all signatures: 
in the Riemannian case and in all dimesnions the answer is due to \cite{Levi-Civita1896}. For dimension 2 in the signature (+,-) 
the answer was almost known to Darboux \cite[\S\S  593,  594]{Darboux1896}, see the discussion in \cite{BolsinovMatveevPucacco2009}. The general case (all dimensions, all signatures) was done in  \cite{BolsinovMatveev2015}.  \end{remark}

\noindent{\bf Proof of Theorem \ref{thm:dini}. } By Remark \ref{rem:conformal} in a neighborhood of a generic points there are coordinates such that  
 the metrics $g$ and $\bar g$ are diagonal.  Indeed, at the points where $g$ is not proportional to $\bar g$ 
  the (1,1)-tensor $ g^{-1}\bar g  = g^{is} \bar g_{js}$  has two different eigenvalues. We consider the coordinate system $(x,y)$ such that $\tfrac{\partial }{\partial x}$ and $\tfrac{\partial }{\partial y}$ are eigenvectors.    
  Since the 
   eigenvectors are orthogonal w.r.t. $g $ and w.r.t. $\bar g$, in this coordinates the metrics are diagonal.  
In this coordinate system, the  corresponding solutions  $$\sigma = \left(g^{ij} \otimes \left(  \operatorname{Vol}_g\right)^{\tfrac{2}{n+1}}\right)= g^{ij} (\det g)^{\tfrac{1}{n+1}}\, , \ \bar \sigma= \left(\bar g^{ij} \otimes \left(  \operatorname{Vol}_{\bar g}\right)^{\tfrac{2}{n+1}}\right)= \bar g^{ij} (\det \bar g)^{\tfrac{1}{n+1}}$$ of the metrisability   equation  are diagonal as well. 

\smallskip 
   Consider the (1,1)-tensor  field \begin{equation} \label{eq:18} A= \bar \sigma (\sigma)^{-1}:= 
\bar \sigma^{is}\sigma_{js}, \end{equation}   where $\bar \sigma^{is}$ is the dual weighted tensor to $\bar \sigma$, i.e.,  
$\bar \sigma^{is}\bar \sigma_{js}= \delta^i_j$. It is a symmetric (0,2)-tensor of projective weight (-2).  In our coordinate system,  it is also  diagonal:
\begin{equation} 
\sigma= \begin{pmatrix} \sigma^{11} &   \\
& \sigma^{22} 
\end{pmatrix} \ , \ \ A= \begin{pmatrix}{    A_1} &   \\
& {   A_2} 
\end{pmatrix} \ , \  \  \bar\sigma= \begin{pmatrix} {   A_1} \sigma^{11} &   \\
&{    A_2}\sigma^{22} 
\end{pmatrix}.  \label{zvezda} 
\end{equation} 
Note that $A$  is indeed a tensor field, since the inverse weighted tensor  $(\sigma)^{-1}$ has weight (-2), and so the weights of $\bar \sigma$ and of $(\sigma)^{-1}$  cancel each other.  In the terms of metrics $g$ and $\bar g$ the tensor $A$ is given by 
\begin{equation}
\label{L}
A_j^i := { \left|\frac{\det(\bar g)}{\det(g)}\right|^{\frac{1}{n+1}}} \bar g^{ik}
 g_{kj}  
\end{equation}
(in the present  section $n=2$  but later the formula   will be used in   all dimensions). 
  
  Let us now plug  $\sigma$ and $\bar \sigma$ from \eqref{zvezda}  in the equations in the metrisability  Theorem \ref{thm:4} whose two-dimensional version is  \eqref{metrization}. We obtain a system of 8 PDE on the 8   unknown functions: the unknown  functions   $\sigma^{11}, \sigma^{22}, A_1,A_2,$ come with their 1st derivatives in the system, and the unknown functions $ K_0,K_1,K_2, K_3$ come as coefficients:

 $$
\left. \begin{array}{ccc}
{\sigma^{22}}_x-\tfrac{2}{3}\,K_1\,\sigma^{22}  &=&0  \\
{\sigma^{22}}_y 
-\tfrac{4}{3}\,K_2\,\sigma^{22} +2\,K_0\,\sigma^{11}&=&0   \\
 {\sigma^{11}}_x
-2\,K_3\,\sigma^{22} +\tfrac{4}{3}\,K_1\,\sigma^{11}&=&0   \\ 
{\sigma^{11}}_y+ \tfrac{2}{3}\,K_2\,\sigma^{11} &=&0\label{4}\end{array}\right| \ \ \ \begin{array}{ccc}
{   A_2} {\sigma^{22}}_x + {   (A_2)_x} {\sigma^{22}} -\tfrac{2}{3}\,K_1\, {   A_2}\sigma^{22}  &=&0  \\
{   A_2} {\sigma^{22}}_y +  {   (A_2)_y} {\sigma^{22}} 
-\tfrac{4}{3}\,K_2\, {   A_2} \sigma^{22} +2\,K_0\, {   A_1} \sigma^{11}&=&0   \\  
{    A_1}{\sigma^{11}}_x + {    (A_1)_x}{\sigma^{11}}
-2\,K_3\, {   A_2}\sigma^{22} +\tfrac{4}{3}  \,K_1\,{   A_1}\sigma^{11}&=&0 \\
{   A_1 } {\sigma^{11}}_y+  {   (A_1)_y}\sigma^{22} +  \tfrac{2}{3}\,K_2\, {   A_1}\sigma^{11} &=&0.   
\end{array} $$

  It is easy to solve the system:  solve the  first 4 questions with respect to $K_0,...,K_3$  (which is a easy linear algebra)  and substitute the result in the last 4 equations. One obtains the equations

$$\left(\begin{array}{ccc}
(A_1)_y&=&0\\
(A_2)_x&=&0\\
((A_1 -A_2)\sigma^{11} (\sigma^{22})^2)_x&=&0\\ 
((A_1 -A_2)\sigma^{22} (\sigma^{11})^2)_y&=&0.
\end{array}\right).   $$ 
We clearly see that the first two  equations imply  that  $A_1= X(x)$ and  $A_2= Y(y)$  for some functions $X$, $Y$  of the indicated variables. Plugging these into the last two equations, we obtain   
$$(X(x) -Y(y))\sigma^{11} (\sigma^{22})^2=\tfrac{1}{Y_1(y)} \  \textrm{and} \   
(X(x) -Y(y))\sigma^{22} (\sigma^{11})^2 =\tfrac{1}{X_1(x)}.  $$

Observe now, because of \eqref{eq:sigma}  and because of the matrices $\sigma, \bar \sigma$ are diagonal, 
 we have  $ \sigma^{11} (\sigma^{22})^2= g^{22}$ and  $ \sigma^{22} (\sigma^{11})^2= g^{11}$. Thus, we obtain  that 
$$
g= (X-Y) (X_1 dx^2 +Y_1 dy^2)  \ \textrm{and}  \ A= \diag(X, Y).   
$$

By a coordinate change $x= x(x_{new}), y= y(y_{new})$, one can ``hide'' $X_1$  and $Y_1$ in $dx^2$ and  $dy^2$ and obtain the formulas (\ref{dini1},\ref{dini2}) of Dini,  \qed

\begin{remark} 
In the multidimensional Riemannian case the proof  is essentially the same, but requires some additional work that should be invested to show the existence of the ``diagonal'' coordinates.  The case  of metrics of arbitrary signature is essentially more complicated.
\end{remark}

 \section{ Tensor invariants of the projective structure and proof of Beltrami Theorem.}  \label{sec:4}

\subsection{Definition and examples}

  {\it Tensor invariants} of a projective structure 
 are tensor fields canonically constructed by an  affine connection  such that they  do
  not depend on the choice of affine connection within a  projective structure.

\noindent{\bf Not an example:  } Curvature  and Ricci tensors are  NOT  tensor invariants. 
Indeed, if we  replace a connection $\Gamma$ by the (projectively equivalent) connection $\bar \Gamma$ given by (\ref{ast}),
 then the direct calculations using  the straightforward formula 
$$
    R^m_{\ \,ikp} = \partial_k \Gamma^m_{\ \,ip} - \partial_p \Gamma^m_{\ \,ik} + \Gamma^a_{\ \,ip} \Gamma^m_{\ \,ak} - \Gamma^a_{\ \,ik} \Gamma^m_{\ \,ap} 
$$
give us the following relation between the curvature tensors of $\Gamma$ and  $\bar \Gamma$:
\begin{equation} \label{eq:curv} 
\bar R^h_{\  i j k} = R^h_{\ ij k}  + { \left(\phi_{j,k}- \phi_{k,j}\right)}  \delta^h_{\ i} + \delta^h_{ \ k} \left( \phi_{i,j}- \phi_i \phi_j \right) -  \delta^h_{ \ j} \left( \phi_{i,k}- \phi_i \phi_k\right).
\end{equation}
  Contracting this formula with respect to $h, k$, we obtain the following relation of the Ricci curvatures of $\Gamma$ and $\bar \Gamma$: 
\begin{equation} \label{eq:ric} 
\bar R_{ij}=  R_{ij} + (n-1)\left(\phi_{i,j}- \phi_i \phi_j\right)  + { \phi_{i,j} - \phi_{j,i}}.
\end{equation}

Though neither   curvature  tensor nor   Ricci  tensor are projective invariants,   one can cook a projective invariant  with their help; it was done in \cite{Weyl1921}:

 \begin{theorem} \label{thm:weyl} Projective Weyl tensor given by the formula \begin{equation} \label{weyltensor}
W^h_{\ ijk}= R^h_{\ ijk} - \tfrac{1}{n-1} \left(\delta^h_{\ k}  R_{ij}   - \delta^h_{\ j}  R_{ik}\right)
+ \tfrac{1}{n+1} \left(\delta^h_{\ i} R_{{ [jk]}}   - \tfrac{1}{n-1} \left( \delta^h_{ \ k} R_{{ [ji]}} - \delta^h_{\ j}R_{{ [ki]}} \right)   \right).\end{equation} 
   is a tensor invariant of a projective structure.
\end{theorem} 
 
\vspace{1ex} 

\begin{proof}
Substituting the formulas \eqref{eq:curv} and \eqref{eq:ric} in  \eqref{weyltensor} we see that all terms containing $\phi$ cancel.\end{proof}

Note that in the most interesting situations the Ricci tensor is symmetric; for example, it is always   the case if our connection is a Levi-Civita connection.  If the  Ricci tensor is symmetric, the second bracket from the right hand side of \eqref{weyltensor} vanishes, and the formula for the Weyl tensor becomes more easy (see \eqref{LCweyl} below).

In dimension 2, Weyl tensor is necessary identically zero, since  each $(1,3)$  tensor  with its symmetries  is zero. 
Fortunately and exceptionally,
 there is one more tensor invariant in dimension 2:

 \begin{theorem}[\cite{Liouville1889}] \label{li} In dimension 2, the tensor field 
\begin{equation*}
L = (L_1\, d x + L_2\, d y)\otimes ( d x \wedge  d y), 
\end{equation*}

where  \begin{equation}\label{eq: Liouvilleinvariants} 
\begin{aligned}
L_1&= 2{K_1}_{xy}-{K_2}_{xx}-3{K_0}_{yy}- 6K_0{K_3}_x- 3K_3{K_0}_x*\\
&\qquad + 3K_0{K_2}_y + 3K_2{K_0}_y + {K_1}{K_2}_x - 2{K_1}{K_1}_y\\
L_2&= 2{K_2}_{xy} - {K_1}_{yy} - 3{K_3}_{xx} + 6K_3{K_0}_y + 3{K_0}{K_3}_y*\\
&\qquad - 3K_3{K_1}_x - 3K_1{K_3}_x - {K_2}{K_1}_y + 2{K_2}{K_2}_x*
\end{aligned}
\end{equation}
 is a tensor invariant   of the  projective structure.  
 \end{theorem} 
 \vspace{2ex}

 From the formulas for $L_1$ and $L_2$  above it is not evident that $L$ is a tensor field; but it is the case. A geometric sense of $L$ is explained in \cite{Cartan1924}. For the goals of our paper, it is sufficient to restrict ourself  to the metric case (when our connection is the Levi-Civita connection of a metric). In this case by direct calculations we see that up to a constant coefficient \begin{equation} \label{eq:liouv} L_{ijk} = R_{ij,k} - R_{ik,j}, \end{equation} 
 and  in this restricted case one   proves Theorem \ref{li}   similar to Theorem \ref{thm:weyl}:
 substituting the formulas \eqref{eq:curv} and \eqref{eq:ric} in  \eqref{eq:liouv} we again  see that all terms containing $\phi$ cancel. 
  \begin{remark} There is a similar story in conformal geometry: conformal Weyl tensor  $C^i_{\ jk\ell} $ vanisihes for $\textrm{dim}(M) \le 3$ but in dimension 3 there  exists an additional conformal invariant and in dimension 2 conformal geometry is not interesting  all.  There is a deep  explanation of this similarity, in fact both conformal and projective geometries are parabolic geometries,  and there are many results in the $n+1$ dimensional conformal geometry that are visually similar to results in the n-dimensional projective geometry (see e.g. \cite{CapSlovak2009}); we will not discuss it here   but  we mention     that many ideas from this paper can be effectively used in the conformal geometry as well. 
\end{remark}

\subsection{ Application of the projectively-invariant tensors: proof of Beltrami Theorem. }

\begin{proposition} \label{prop:dimn}  Let $\nabla^g=\left( \Gamma^i_{jk}\right)$ be the Levi-Civita connection of $g$ on  $M$ with $n=\textrm{dim} (M)>2$. Then, $W^h_{\ ijk}\equiv 0$ if and only if  $g$ has constant sectional curvature. 
\end{proposition}

\begin{proof}
  For Levi-Civita connections the Ricci tensor is symmetric so the formula for $W$ reads
\begin{equation}\label{LCweyl}
W^h_{\ ijk}= R^h_{\ ijk} - \tfrac{1}{n-1} \left(\delta^h_{\ k}  R_{ij}   - \delta^h_{\ j}  R_{ik}\right) .\end{equation}

  If $W\equiv 0$, we obtain 
$$
R^h_{\ ijk} = \tfrac{1}{n-1} \left(\delta^h_{\ k}  R_{ij}   - \delta^h_{\ j}  R_{ik}\right). 
$$
  After lowing the index we have therefore 
$$
R_{hijk} = \tfrac{1}{n-1} \left(g_{h k}  R_{ij}   - g_{hj}  R_{ik}\right).
$$
  We see that the left-hand-side is symmetric with respect to $({   h,i},{ j,k}) \longleftrightarrow ({ j,k},{   h,i})$, so should be the right-hand-side, which implies that 
$R_{ij}$ is proportional to $g_{ij}$,  $R_{ij}=\tfrac{R}{n}g_{ij}$  so we have
 $$
R_{hijk} = \tfrac{R}{n(n-1)} \left(g_{h k}  g_{ij}   - g_{hj}  g_{ik}\right)
$$
which is equivalent to ``sectional curvature is constant''. \ 
  \end{proof}

A similar statement is valid in dimension 2:

\begin{proposition} \label{prop:dim2}   Let $\nabla^g=(\Gamma^i_{jk})$ be the Levi-Civita connection of $g$ on 2-dim $M$. Then, $L_{ijk}\equiv 0$ if and only if  $g$ has constant curvature.   \end{proposition}

\begin{proof}
It is well-known (and follows  from the symmetries of the curvature tensor) 
 that the 2-dim manifold are automatic Einstein in the sense that
 
 $$
 R_{ij}= \tfrac{1}{2} Rg_{ij}.  
 $$ 
Calculating $L_{ijk}$ gives 

$$
L_{i{ jk}} =  R_{i{ j,k}} - R_{i{ j,k}}= \tfrac{1}{2} \left( R_{,{ k}} g_{i{ j}}-   R_{,{ j}} g_{i{ k}}\right) .
$$
Since $g$ is nondegenerate,
 vanishing of $L$ implies vanishing of $R_{,k}$ and hence the constancy of the curvature. 

\end{proof}

Combining Propositions \ref{prop:dimn} and  \ref{prop:dim2}, we obtain the following statement:

\begin{corollary}[Beltrami Theorem; \cite{Beltrami1865} for dim $2$; \cite{Schur1886} for dim$>2$; see \cite{Eastwood2017, DiScala2005,Matveev2006b} for alternative  proofs] 
A metric projectively equivalent to a metric of constant curvature has constant curvature. 
\end{corollary}

\section{Projective transformations and  Lichnerowicz-Obata conjecture. }  

\noindent{\bf Definition. }   {\it Projective transformation } of a projective structure $[\Gamma]$
is a   diffeomorphism that preserves $[\Gamma]$. 

\smallskip 
\noindent{\bf Geometric (equivalent) definition}. Projective transformations are diffeomorphisms
 that  send geodesics of $[\Gamma]$ to geodesics. In this definition we consider geodesics up to reparameterization.

 \begin{example}[Beltrami example]    \label{BE}
  We consider  the standard sphere  $S^n \subset R^{n+1}$ with the induced metric and its Levi-Civita connection.   
  Then, for every    $A\in SL(n+1)$  the   diffeomorphism  
  $$ a:S^n\to S^n,   a(x):=\frac{1}{|Ax|} Ax$$
  is a projective transformation of the sphere. 
  (In the formula above $Ax$ means multiplication of the $((n+1)\times(n+1))$-matrix  $A$ with $x\in \mathbb{R}^{n+1}$, and $|Ax|$ means the usual Euclidean length of $Ax\in \mathbb{R}^{n+1}$. Note that the length of $\frac{1}{|Ax|} Ax$ is 1 so it does lie on the sphere).\end{example}

Indeed, geodesics of  the sphere are the great circles, that are the intersections of the 2-planes containing the center of the sphere with the sphere. Since multiplication with $A$  is a linear bijection, the image of a 2-plane  containing the center of the sphere is a  2-plane containing the center of the sphere, so $a$ sends the intersection of the sphere with the first plane to the intersection of the sphere to the second plane.

Clearly, all projective transformation of a given manifold form a Lie group which we denote $\operatorname{Proj}$. It has dimension at most $(n+1)^2-1= n^2+2n$. The group of affine (i.e., connection-preserving) transformations will be denoted by $\operatorname{Aff}$, and the group of isometries is $\operatorname{Iso}$. Clearly, $\operatorname{Proj}\subseteq\operatorname{ Aff}  \subseteq \operatorname{ Iso}$, and $\operatorname{Aff}$ is a normal subgroup of $\operatorname{Proj}$ and $\operatorname{Iso}$ is a normal subgroup of  $\operatorname{Aff}$.

In this section we will discuss and give the answer to the following 

\vspace{1ex} 
\noindent{\bf Natural question.}  {\it How big can be the quotient group $\operatorname{Proj/Aff}$ for a complete Riemannian manifold $M^n$ (with $n\ge 2$)}? 

\vspace{1ex}

Beltrami example above shows that for the standard sphere the quotient  $\operatorname{Proj/Aff}$ is relatively big  and in particular contains infinitely many elements. For certain quotients of the standard sphere $\operatorname{Proj/Aff}$  also contains infinitely many elements. The next theorem, which is the main new result of the present paper,  says that on other manifolds the quotient group  $\operatorname{Proj/Aff}$  is actually finite and contains at most two elements.

\begin{theorem} \label{thm:new}  Let $(M,g)$ be a complete  Riemannian manifold of dimension  $n\ge 2$ 
such that the sectional curvature is not  a positive constant. Then, $\operatorname{Proj/Aff}$ contains at most two elements.
\end{theorem}

\begin{remark}  For closed manifolds, Theorem \ref{thm:new} was proved in \cite{Matveev2014}; as it is clearly explained there, 
essential part of the  proof is actually due to \cite{Zeghib2013},  who has proved that  $\operatorname{Proj/Aff}$ contains at most $2n$ elements.    
\end{remark}

\begin{remark} Theorem \ref{thm:new} is a stronger version of a famous conjecture due to Lichnerowciz and Obata: they conjected  that  a connected group of projective transformations on a closed (Obata) or complete (Lichnerowicz) Riemannain 
manifold of nonconstant curvature consists of affine transformations. This conjecture was proved in \cite{Matveev2005} for dimension 2 and in \cite{Matveev2007} for other dimensions, the latter reference contains also a description of the history of the problem including a list of previous results in this direction. 
\end{remark}

\begin{remark} Theorem \ref{thm:new} is sharp: 
there exist examples (see e.g. \cite{Matveev2014}) of closed and complete manifolds such that $\operatorname{Proj/Aff}$ contains precisely  two elements, and its natural generalisations 
 are  not true locally (see e.g. \cite{BryantMannoMatveev2008, Matveev2012b, KruglikovMatveev2014}. It is not clear though whether the assumption that the metric is positive definite is important: we do not have counterexamples and  recently projective Obata conjecture (on closed manifolds and with connected groups) was proved for metrics of Lorentzian signature  in \cite{Matveev2012c} for dimension 2 and in  \cite{BolsinovMatveevRosemann2015} for all dimensions. 
\end{remark}

 Theorem \ref{thm:new} will be proved in the next two sections.

\subsection{ Space of solutions of the metrisability equation and how projective transformations  act on it.} In this and in the following section 
we assume that $g$ is a complete   Riemannian metric on a connected $M^n$ of dimension $n\ge 2$; we assume that the sectional curvature of $g$ is  not a positive constant   and our goal is to show  that the number of elements in the quotient group $\operatorname{Proj/Aff}$ is at most two.

We consider  the metrisability  equation \eqref{5}. 
Let $\operatorname{Sol}$ be the space of its solutions; since the equation is linear, it is a linear vector space. It has a  finite dimension.  

The following theorem, which was proved in \cite[\S4 ]{Matveev2005} for  dimension 2 and in \cite[Theorem 1]{KiosakMatveev2010}  for dimensions  $\ge 3$ 
(in fact, in the Riemannian case,  \cite[Theorem 2]{Matveev2006a} or \cite[Theorem 16]{Matveev2007}  are sufficient), plays  a crucial role.

\begin{theorem} \label{thm:degree} 
Let $g$ be a complete Riemannian metric  on a connected  $M^{n}$ of dimension  $n\ge 2$.  Assume that $g$ does not have constant positive sectional curvature. 

If the dimension of $\operatorname{Sol}$ is not equal to  $2$, then  every complete metric $\bar g$ projectively  equivalent to $g$ is affine equivalent to $g$. 
\end{theorem}  
 
 We do  not give or explain  the proof of this theorem. It is pretty involved and is based on another group of methods than that used in this paper; moreover, the proofs in dimensions $n= 2$   and $n\ge 3$ are very different.

By Theorem  \ref{thm:degree}, in the proof of Theorem \ref{thm:new},  we may assume without loss of generality that $\textrm{dim}(\operatorname{Sol})= 2$.

  Consider now a projective transformation  $\phi \in \operatorname{Proj} $. 
  Since the equation  \eqref{5} is projectively invariant, $\phi$ sends solutions of the metrisability  equation to solutions.
  
  \begin{remark} In order to construct weighted tensor bundles, we normally need to fix an orientation of the manifold.  Though our projective transformation are not assumed to be  orientation-preserving,  
   since  solutions of the metrisability  equation  have  the even projective weight $(-2)$, they do not depend on the orientation at all and no problem appears.
\end{remark}

  Take a basis $\sigma, \bar \sigma $ in $\operatorname{Sol}$ and consider the pullbacks $\phi^*\sigma$, $\phi^*\bar \sigma$. They also belong to $\operatorname{Sol}$ and are therefore linear combinations of the basis solutions  $\sigma$ and $\bar \sigma$; we denote the coefficients as below: 
$$
\begin{pmatrix} \phi^*\sigma \\ \phi^*\bar \sigma\end{pmatrix} = \begin{pmatrix} a & b \\ c & d \end{pmatrix} \begin{pmatrix} \sigma\\ \bar \sigma\end{pmatrix} = \begin{pmatrix} a\sigma  &+ & b\bar \sigma  \\ c \sigma  &+& d\bar \sigma \end{pmatrix}.
$$ 
We denote the matrix $\begin{pmatrix} a & b \\ c & d \end{pmatrix}$ above   by $T_\phi$.  
The  mapping  from $\operatorname{Proj}$ to $GL(2, \mathbb{R})$ given by $\phi\mapsto T_\phi$ 
is actually a 2-dimensional  representation of $\operatorname{Proj}$, since 
the composition  $\psi\circ \phi$ of two projective transformations corresponds to the product of matrices $A_\psi $ and $A_\phi$ is the reverse order:
$$
\psi\circ \phi \mapsto T_\phi T_\psi. 
$$

 It is easy to see that  if a metric $g$ from the projective class corresponds  to $\tilde \sigma \in \operatorname{Sol}$ and  if $\tilde \sigma$ is an eigenvector of $\phi^*$, then $\phi$ is a homothety  for $g$.  As a consequence we  have that if $T_\phi= \operatorname{id}$    then $\phi$ is an  isometry w.r.t. any metric in the projective class.

As we explain below,  Theorem \ref{thm:new}  immediately follows from the next  proposition:

\begin{proposition} \label{prop:main} We assume that $g$ is a complete   Riemannian metric on a connected $M^n$ of dimension $n\ge 2$ such   that its  sectional curvature of $g$ is  not a positive constant and such that it admits a  complete metric which is projectively equivalent to $g$ and which is not affine equivalent to $g$.  Then, any  projective transformation $\phi$ such that 
 the determinant of 
 $T_\phi$ is positive  is a homothety of $g$. 
\end{proposition}

Proposition \ref{prop:main} clearly implies Theorem \ref{thm:new}. Indeed, by Proposition \ref{prop:main} for two 
 nonhomothetic   projective transformations $\phi$ and $\psi$     there superposition is 
   a homothety and hence an affine transformation, since the product of two matrices $T_\psi$  and $T_\phi$
  with negative determinants has positive determinant. Thus, product of two arbitrary elements of the quotient group $\operatorname{Proj/Aff}$ is a  trivial element, and the number of elements in $\operatorname{Proj/Aff}$ is at most two.

We will start the proof of   Proposition  \ref{prop:main} now, explain  the scheme and  prove  two simple cases in this section.  The most involving case  will require preliminary work and will be proved in the next section. 

First of all,  suppose for a  nonhomothetic projective transformation  $\phi$ the determinant of $T_\phi$ is positive. Then, by  a choice of a  basis $\sigma, \bar \sigma$ in $\operatorname{Sol}$ we achieve  that $\phi^*$  is as in one of the three cases below:  
\begin{equation} \label{matrices}
  {\left[\begin{array}{cccr}\phi^* \sigma &=&  c\, \sigma &     \\
                  \phi^* \bar \sigma & = & & \bar c \, \bar \sigma\end{array}\right] }  \hspace{3ex} {    
                   \left[ \begin{array}{ccr}\phi^* \sigma &=  C \cos(\alpha)\,   \sigma & - C \sin(\alpha)\, \bar \sigma   \\
                  \phi^* \bar \sigma & =  C  \sin(\alpha)\,    \sigma & + C \cos(\alpha)\,  \bar \sigma\end{array}  \right] } 
                  \hspace{3ex} \left[ \begin{array}{ccr} \phi^* \sigma &= & \lambda \, \sigma  +  \bar \sigma   \\
                  \phi^*\bar \sigma & = & \lambda \, \bar  \sigma\end{array}\right]. 
\end{equation} 
The parameters $c, \bar c, \lambda, C$ above     are real numbers different from zero such that    $ c \, \bar c>0$; we  may  assume that  $C>0$.

In this section we show 
 that the second and third cases   of \eqref{matrices} are impossible. We first consider the  second case,  it  is slightly more complicated than the third one. 
 The most complicated case is actually the first one, we will consider it in the next section.

Consider the superposition $\phi^m= \underbrace{\phi \circ ....\circ \phi}_{m}$. It is a projective transformation and the corresponding matrix is simply the $m^{\textrm{th}}$ power of the matrix  $C \begin{pmatrix} \cos \alpha & -\sin\alpha   \\ \sin\alpha & \cos \alpha \end{pmatrix} $, i.e., the matrix 
$$C^m \begin{pmatrix} \cos (m \alpha) & -\sin(m\alpha)   \\ \sin(m\alpha) & \cos (m\alpha) \end{pmatrix}.  $$

Suppose  the metric $g$ corresponds, via the formula \eqref{eq:sigma}, to a linear combination $\tilde \sigma= a\sigma + b \bar \sigma$.
Let us now show   
that, unless $\alpha\ne 2 \pi N$  for an integer $N$, there exists $m$ such that  ${\phi^{m}}^*( \tilde \sigma)$ is not positive definite. 

 We first assume that  $\tfrac{\alpha}{2\pi}$ is irrational. Then for a certain   $m$ the matrix    
$\begin{pmatrix} \cos (m \alpha) & -\sin(m\alpha)   \\ \sin(m\alpha) & \cos (m\alpha) \end{pmatrix}$ is very close  to  the matrix 
$-\operatorname{id}$. Indeed, the matrices $\begin{pmatrix} \cos (m \alpha) & -\sin(m\alpha)   \\ \sin(m\alpha) & \cos (m\alpha) \end{pmatrix}$ correspond to $m\alpha$-rotation around the zero points, so the points of the form 
 $\begin{pmatrix} \cos (m \alpha) & -\sin(m\alpha)   \\ \sin(m\alpha) & \cos (m\alpha) \end{pmatrix}$ generate an everywhere dense subset in the group $SO(2)$, which implies that there exists $m$ such that the matrix $ \begin{pmatrix} \cos (m \alpha) & -\sin(m\alpha)   \\ \sin(m\alpha) & \cos (m\alpha) \end{pmatrix}$ is very close to $-\operatorname{id}$; for this matrix the pullback of  $\tilde \sigma$ which is  $-C^m \tilde \sigma$   is  negative definite.  But this is impossible since this would imply that the pullback of a positive definite $g$ is negative definite. The contradiction shows that this case is impossible.

 Suppose now  $\tfrac{\alpha}{2\pi}$ is  rational, but not  integer: $\alpha= 2\pi \tfrac{p}{q}$ with $q\ge 2$.   Then, the sum of solutions 
 $\tilde \sigma + \tfrac{1}{C} \phi^* (\tilde \sigma) + ... + \tfrac{1}{C^{q-1}}  {\phi^{q-1}}^*( \tilde \sigma)$ is zero, because the solutions 
 $\tfrac{1}{C^{m}}  {\phi^{m}}^*( \tilde \sigma)$  are vertices of the regular $q$-gone in the place $\operatorname{Sol}$, and 
 the sum of the vertices of a regular polygon with center at  origin is zero.  But then $\tilde \sigma$ can not be positive definite since a sum of positive definite matrices is positive definite as well and can not be zero. Thus also this case is impossible. 
 
 Finally, $\alpha= 2 \pi N$, so our $A_\phi= \begin{pmatrix} C & \\ & C \end{pmatrix}$; hence $\phi$ is a homothety for any metric in the projective class and therefore    for $g$.   
Thus, the second case of \eqref{matrices}
 is impossible. 
 
Similarly, one can  show that the third matrix of  \eqref{matrices} is impossible. Indeed, 
in this case   the superposition $\phi^m$ corresponds to the matrix 
$$
 \begin{pmatrix} \lambda & 1   \\ 0 & \lambda \end{pmatrix}^m = \begin{pmatrix} \lambda^m &  m\lambda^{m-1}   \\ 0 & \lambda^m \end{pmatrix}.   
$$
It is easy to see that unless $\tilde \sigma$ is an eigenvector of $\phi^*$ (which implies that $\phi$ is a homothety, we explained this above), 
for big $m$ the solution $\tfrac{1}{\lambda^m}{\phi^{m}}^*( \tilde \sigma)$  is close to  $\tfrac{m}{\lambda}  \bar \sigma$,   the solution ${\lambda^m}{\phi^{-m}}^*( \tilde \sigma)$  is close to  $\tfrac{-m}{\lambda}  \bar \sigma$,   and it is not possible that both of them are positive definite.

 Thus, below we can assume that $\phi$ is as in the first case of   \eqref{matrices}. Note that the case  $c=\bar c$, corresponds to 
   homothety, which is of course an affine transformation;  we may  assume therefore  $ c \ne \bar c $.

\section{Proof  of Proposition \ref{prop:main} in  the remaining case, and hence of Theorem \ref{thm:new} }

\subsection{Integrals in the multidimensional case and complete manifolds such that $A$ has two constant and one nonconstant eigenvalue.}  \label{intmul} We have seen in Theorem \ref{thm:int} that  the existence of $\bar g$ projectively equivalent to $g$ allows one  to construct an integral of the geodesic flow of $g$. Since the metrisability  equation is linear, if we have one metric that is projectively equivalent to $g$ and nonproportional to $g$, 
then, at least locally, 
 we have a two-parameter 
  family   $C \cdot g_s$ of metrics projectively equivalent to $g$; so we have a 1-parameter family of the integrals. Direct calculations give us the following theorem:

\begin{theorem}[\cite{BolsinovMatveev2003,MatveevTopalov1998,TopalovMatveev2003}] Let $g$   be a metric of arbitrary signature. 
    Consider the (1,1)-tensor  $A$ given by  \eqref{L} (assuming that $\bar g$ is projectively equivalent to $g$) or by \eqref{eq:18} (assuming that $\sigma$ is the solution of the metrisability  equation corresponding to $g$ and   $\bar \sigma $ is also a solution of the metrisability  equation). 
    
Then for any $t\in \mathbb{R}$, the function

\begin{equation} \label{integral} I_t:TM\to \mathbb{R},  \quad  I_t(\xi)=  g(\operatorname{co}(t\cdot \operatorname{id}-A) \xi, \xi), 
  \end{equation}where $\operatorname{co}$ denotes  the comatrix (which is this case is a $(1,1)$-tensor),  is an integral of the geodesic flow of the metric $g$.
   \end{theorem}

Clearly, the entries of the comatrix are polynomial expressions of order $n-1$ in the entries of the initial matrix, so the family of integrals  $I_t$ is actually a polynomial in $t$  of degree $n-1$, whose coefficients are integrals.

\begin{corollary}[\cite{GoverMatveev2015},Corollary 5.7] The minimal polynomial of $A$ has the same degree at each point of an open
everywhere dense subset of $M.$ \label{gover}
\end{corollary} 

We will not prove this corollary here, and refer the reader to \cite[\S 5.3]{GoverMatveev2015}. Note that   in the Riemannian case (and later we will need this corollary only in the Riemannian case), Corollary was proved in \cite{BolsinovMatveev2003,TopalovMatveev2003}.  
Let us  mention here that the proof of this corollary uses the same ideas as the proof of Theorem \ref{thm:5}, namely 
 the special form of the integrals (they are quadratic in the velocities and the family of the integrals depends  polynomially on $t$) and the condition that the integrals are  constant on geodesics. We will play with the same ideas below, in the proof of Proposition \ref{twocomponents}.

Let us now consider the following special case which will play principal role in the proof of the remaining part of Proposition \ref{prop:main}: \begin{enumerate} \item \label{item1} 
assume our manifold $(M,g)$ is complete and simply connected, \item  assume $\bar \sigma$ is 
 a solution of the metrisability  equation (for the Levi-Civita connection of $g$)  \item  such that the tensor $A$ given by \eqref{eq:18}, where $\sigma$ is the solution of the metrisability equation corresponding to the metric $g$, 
 at the  generic point has the following structure of eigenvalues:  it has three eigenvalues: $0$ (of multiplicity $m$), $\lambda$ (of multiplicity $1$) and $1$ (of multiplicity $\bar m$).  \label{itemlast}\end{enumerate} 
 
 We allow that $m$ or $\bar m$ are zero.  Clearly,   $\lambda$ is a smooth function, and $n=m+\bar m+1$. 
 
 Now, denote by $M_0$ resp. $M_1$ the sets  $M_0:= \{ p\in M \mid \lambda(p)= 0\}$ and  $M_1:= \{ p\in M \mid \lambda(p) = 1\}$. 
Our goal is to prove the following proposition: 
 
 \begin{proposition} \label{twocomponents} 
 Under the assumptions (\ref{item1}-\ref{itemlast}) above, $M_0$ (resp. $M_1$) is either empty or a smooth totally geodesic geodesically complete submanifold of dimension $\bar m$  (resp. $m$) which has at most two connected components.  

Moreover, for any  point $\gamma(s)$ of any  geodesic $\gamma$ orthogonal to $M_1$ or to $M_0$  the  velocity  vector  $\dot \gamma(s)$  is a  $\lambda$-eigenvector of $A$.   
 \end{proposition} 
 
 In fact, from the proof it will be clear that the total 
 number of connected components of $M_0\cup M_1$ is at most two (so if for example $M_0$ has two connected components then $M_1$ is empty).

 {\bf Proof.} Clearly $M_0$ and $M_1$ interchange if we replace $\bar \sigma$ by $\sigma-\bar \sigma$, so it is sufficient to prove Proposition for $M_0$ only. 
 
  By an \emph{adapted frame} at a point $p\in M$ we understand a basis $v_1,...,v_n$ in $T_pM$ such that in this basis the matrix of $g$ is the identity matrix, and the matrix of $A$ is diagonal such that the first diagonal element  
 is $\lambda$, the next $m$ diagonal elements are $0$, and the remaining $\bar m$ diagonal elements are $1$:
 \begin{equation}\label{diagonal-1}
g= \diag(1,1,...), \   A= \diag(\lambda, \underbrace{0,...,0}_m,\underbrace{1,...,1}_{\bar m} ).
\end{equation}
  The existence of such a basis follows from linear algebra. We see that the first vector $v_1$ is an eigenvector  of $A$ with 
  eigenvalue $\lambda$, the next vectors $v_1,...,v_{m+1}$ are eigenvectors with eigenvalue $0$, and the last vectors $v_{m+2},...,v_{n}$ are eigenvectors with eigenvalue $1$. 
 
 We assume that $m\ge 1$ and $\bar m\ge 1$; the cases when $m$ or $\bar m$ are zero can be handled by the same methods and are easier. 
 In the adapted frame, the family of  integrals $I_t( \xi) $ is given  by 
 $$
 I_t( \xi)= (t-1)^{m-1}t^{\bar m-1}\left( t(t-1)\xi_1^2 + (t-1)(t-\lambda) (\xi_2^2+...+\xi_{m+1}^2)+ t(t-\lambda)(\xi_{m+2}^2+...+\xi_{n}^2)\right). 
 $$
  Since for every $t$ the coefficient  $(t-1)^{m-1}t^{\bar m-1}$ is a constant, we have the following family of the integrals:
 \begin{equation} \label{INT}
 \tilde I_t( \xi)=   t(t-1)\xi_1^2 + (t-1)(t-\lambda) (\xi_2^2+...+\xi_{m+1}^2)+ t(t-\lambda)(\xi_{m+2}^2+...+\xi_{n}^2) . 
 \end{equation}

Suppose  $p\in M_0$. Let us show that locally, near $p$, there exists a submanifold of dimension $\bar m$ containing $p$ and lying in $M_0$. 
In order to do it, consider a sufficiently small $\varepsilon>0$ and the set 
$$S:=\{\xi\in T_pM\mid \xi_1=...=\xi_{m+1}=0, \xi_{m+1}^2+...+\xi_n^2<\varepsilon\}, $$ 
it is an open disc containing zero in the  subspace of dimension $\bar m$. 
Consider all geodesics starting from $p$ with the initial velocity vector lying in $S$. Let us show that all points of all such geodesics $\gamma$ belong to  $M$. 

In order to do it, observe that the family of the integrals \eqref{INT} starting from $p$ with the nonzero  initial velocity vector 
 lying in $S$    is given by $t^2(\xi_{m+2}^2+...+\xi_{n}^2)$; that is, $t=0$ is a zero of order $2$.   Then, the same should be fulfilled at any other point $q$ of such a geodesic. Substituting $t=0$ in \eqref{INT} we obtain 
 \begin{equation}
 0(0-1)\xi_1^2 + (0-1)(0-\lambda) (\xi_2^2+...+\xi_{m+1}^2)+(0-0)(0-\lambda)(\xi_{m+2}^2+...+\xi_{n}^2)=0.  \label{ordering}
\end{equation}
By  way of contradiction, suppose $\lambda\ne 0$. Then,  from \eqref{ordering} we obtain $\xi_2=...=\xi_{m+1}=0$. Substituting this in   \eqref{INT}, we obtain  that the family  of the integrals $\tilde I_t( \xi)$ 
 is 
$$t\left((t-1)\xi_1^2 +   (t-\lambda)(\xi_{m+2}^2+...+\xi_{n}^2)\right).$$ In order $t=0$ be a zero of order $2$, $(0-1)\xi_1^2 +   (0-\lambda)(\xi_{m+2}^2+...+\xi_{n}^2)$ should be zero which is not possible since the velocity vector is not zero. The contradiction shows that at the point $q$ the value of $\lambda$ is again $0$, so the whole geodesic lies in $M_0$. Thus,  in a small neighborhood of $p$  the image of the exponential mapping  of the set $S$, which we denote by $M'_0$,  lies in $M$;  if $\varepsilon$ is small enough, $M'_0$  is  an embedded disk of dimension $\bar m$ lying in $M_0$. 

Let us now show that a sufficiently  small neighborhood $U$ of $p$ does not contain other points of $M_0$ except of those lying in the image of $S$ w.r.t. the exponential mapping. W.l.o.g. we think that  $U$ is geodesically convex; all geodesics considered below are assumed to be contained in $U$.

Take now  a point $q\in U$ such that $q\not \in M_0$, almost every point of $U$ has this property   by Corollary  \ref{gover}. 
Consider a geodesic connecting $q$ with a point $p$ of $M_0$. Since at $p$ we have $\lambda=0$,   we see that for $t=0$ the integral $\tilde I_t$ is zero. The, the same should be true at the point $q$ which gives us \eqref{ordering} which implies that at the point $q$ the tangent vector to the geodesic  lies in the $\bar m+1$-dimensional subspace $S_q \subseteq T_qM$ given by the condition 
$
\xi_2=...=\xi_{m+1}=0.
$

Thus, $M_0$ lies in the image of the exponential mapping  of  $S_q$, which is a $\bar m+1$-dimensional embedded submanifold.

 Take now another point   $q'\in U$ such that   $q'$ does not lie in the image of the exponential mapping of  $S_q$.   The dimension of the manifold implies the existence of such a point, since $m\ne 0$ (otherwise the set $M_0$ is empty and it is nothing to prove) so  $\bar m+1= \operatorname{dim} S_q$ is smaller than $ m+ \bar m + 1= n= \operatorname{dim} M$. Next,  repeat the argumentation we did for $q$  for the point $q'$. 
 We again have that $M_0$ lies in the image of the exponential mapping   of  $S_{q'}$. Since the intersection of the image of the exponential mapping  of  $S_{q'}$ and of $S_q$ is a $\bar m$-dimensional submanifold (at least, if we are working in a very small neighborhood), we obtain that all points of 
  $M_0$  lie in $M'_0$. Thus, $M_0$ is a submanifold of dimension $\bar m$. Locally, it coincides with $M'_0$ which implies that it is totally geodesic. 
  
  Let us now show that $M_0$ contains at most two components, and that 
  the  velocity  vectors  of geodesics  orthogonal to $M_0$ are   $\lambda$-eigenvectors of $A$.    Consider a point $p\in M_0$ and consider a geodesic starting from $p$ with the initial velocity vector orthogonal to $M_0$. As we have seen above, in an admissible frame, the tangent space to $M_0$ is given by the equation   $\xi_1=...=\xi_{m+1}=0$, so the initial velocity vector of our geodesic has 
  $\xi_{m+2}=...=\xi_n=0$. Then, the family of the  integrals  \eqref{INT} is given by 
  $$
  t(t-1)(\xi_1^2  + \xi_2^2+...+\xi_{m+1}^2)
  $$
and we see that $t=0$ and $t=1$ are zeros of it. Then, the same is true at any point of the geodesic. Substituting $t=0$ and $t=1$ in \eqref{INT}, we obtain $$  \lambda(\xi_2^2+...+\xi_{m+1}^2) =0 \textrm{ and } 
    (1-\lambda)(\xi_{m+2}^2+...+\xi_{n}^2)=0.$$ Thus, at the points $q$
    such that $\lambda(q)\not\in\{0,1\}$ we have that the velocity vector of the geodesic such that it orthogonally passes through a point of $M_0$ is an eigenvector corresponding to  $\lambda$. By continuity and since both $M_0$ and $M_1$ are geodesically complete, it is so at all points.

    Suppose now there exists at least two connected component of $M_0$, which we denote by $M_0(1) $ and $M_0(2)$; our goal is to show that $M_0=M_0(1) \cup M_0(2)$. For any point $p$ of $M_0(1)$, 
    consider  the shortest geodesic containing the point $p$ with the points of $M_0(2)$ (the existence is standard and follows from completeness). W.l.o.g. assume there is no other point of $M_0$ on this geodesic  between its startpoint $p\in M_0(1)$ and its endpoint on $M_0(2)$, otherwise replace $M_0(1)$ by  the connected component of that point.   As we have shown above, at every point of this geodesic its velocity vector is an eigenvector of $A$ 
    with eigenvalue $0$. Then, this geodesic is orthogonal to   $M_0(1)$ and $M_0(2)$. Then, each geodesic  starting from a point of $M_0(1)$  orthogonal to $M_0(1)$ comes to $M_0(2)$, and on the way from $M_0(1)$ to $M_0(2)$ it contains no other point of $M_0$. 
    Clearly, the union of all geodesics starting from  $M_0(1)$  orthogonal to $M_0(1)$ covers the whole manifold. Finally, there is simply no place for other connected components of $M_0$. Proposition \ref{twocomponents} is proved.

\begin{remark} \label{rem:byproduct} As a by-product  we have proved that (under the assumptions  that $A$ has three eigenvalues $\lambda$, $0$, $1$ such that $\lambda \ne 0,1$ in a neighborhood we work in), the 1-dimensional eigendistribution corresponding to $\lambda$ is  totally  geodesic.
In the case $M_0$ (resp.   $M_1$) is  not empty, the integral submanifolds of this distribution are precisely the geodesics orthogonal to $M_0$ (resp.  $M_1$). Along these geodesics, points of $M_0$ can be conjugate to points of $M_0$ or $M_1$ only. 

So topologically the manifold is glued from the normal bundles to $M_0$ and to $M_1$ (provided that $M_0\cup M_1$ is not empty. 

\end{remark}

\subsection{ Levi-Civita Theorem  assuming (\ref{item1}-\ref{itemlast}).}  \label{isometry}

Let us now describe the metric $g$ and the (1,1)-tensor $A$, locally and globally, under the assumptions (\ref{item1}-\ref{itemlast}) of \S \ref{intmul}. 
 Locally, it is a special case of what was done  in  \cite{Levi-Civita1896}, we recall it for the convenience of the reader.

\begin{theorem}[\cite{Levi-Civita1896}] \label{LCtheorem} Let $g$ be a metric (of arbitrary signature.  Assume $\bar \sigma$ is 
 a solution of the metrisability  equation (for the Levi-Civita connection of $g$)  
   such that the tensor $A$ given by \eqref{eq:18}, where $\sigma$ is the solution of the metrisability  equation corresponding to the metric $g$, 
 in each point of some neighborhood of $p$  has the following structure of eigenvalues:  it has three eigenvalues: $0$ (of multiplicity $m$), $\lambda\not\in \{0, 1\}$  (of multiplicity $1$) and $1$ (of multiplicity $\bar m$). Then, in some neighborhood of this point there exist coordinates $x_1,x_2....,x_{m+1},x_{m+2},...,x_n$ such that the following conditions are fulfilled:  In this coordinates  $\lambda$ is a function of the variable $x_1$ only, $A$ is a diagonal matrix as below,
 $$
 A= \diag(\lambda(x_1),\underbrace{0,...,0}_m, \underbrace{1,...,1}_{\bar m}),
 $$
 and 
  the metric has the block-diagonal form  below (with one block of dimension $1\times 1$, one block of dimension $m\times m$ and one block of dimension $\bar m\times \bar m$):
\begin{equation}\label{LCformula}
 g= \pm \lambda (1-\lambda) dx_1^2 + (1-\lambda) \sum_{i,j=2}^{m+1} h_{ij} dx_idx_j +  \lambda \sum_{i,j={m+2}}^{n} \bar h_{ij} dx_idx_j,
 \end{equation}
 where the components  $h_{ij}$  are symmetric in $i$ and $j$ and depend on the coordinates $x_{2},...,x_{m+1}$ only, and the components  $\bar 
 h_{ij}$ are symmetric in $i$ and $j$ and depend on the coordinates $x_{m+2},...,x_{n}$ only. 
\end{theorem}

In Theorem \ref{LCtheorem} above we assumed that both $m$ and $\bar m$ are different from zero. In fact, theorem remains correct also without this assumption, but in the case when $A$ has only 1 eigenvalue one can make additional simplifications (e.g., the metric can be brought to the standard warped product form). 

In fact, Levi-Civita proved this statement for positive definite metrics only and we will use this statement for positive definite metrics only; but   his proof is valid for all signatures. Note that the proof of  Theorem \ref{LCtheorem} can be obtained analogous to the proof of Theorem \ref{thm:dini}. 

Note that in the positive definite case the sign $\pm$ before $\lambda (1-\lambda) dx_1^2$ is such that $\pm \lambda (1-\lambda) dx_1^2$
is positive, and  $h$ and $\bar h$ are metrics on $m$ resp. $\bar m$-dimensional manifolds; they are positive of negative 
definite depending on the sign of $(1-\lambda)$  and $\lambda$. 

Let us now consider the geometric sense of the coordinates, and also of $h_{ij}$ and $\bar h_{ij}$.
 Clearly, $\tfrac{\partial }{\partial x_1}$ is a $\lambda$-eigenvector of $A$,  and $\tfrac{\partial }{\partial x_2}$,...,  $\tfrac{\partial }{\partial x_{m+1}}$ are $0$-eigenvectors of $A$, and  $\tfrac{\partial }{\partial x_{m+2}}$,...,  $\tfrac{\partial }{\partial x_{n}}$ are $1$-eigenvectors of $A$. We see that the eigendistribution  of $\lambda, 0, $ and of $1$ are simultaneously integrable. We see also that $h$ and $\bar h$ can be viewed as metrics on integral manifolds of the eigendistribution corresponding to $0$  and to $1$; in fact that up to the coefficient $(1-\lambda)$ and $\lambda$ they are  the restriction of the metric $g$ to these integral manifolds. 

Combining this with Remark  \ref{rem:byproduct}, we see that  if $M_0$ is not empty then 
 the metric $\bar h$ is essentially the restriction of the  metric $g$  to $M_0$ and if $M_1$ is not empty then the metric $h$ is essentially the restriction of the  metric $g$ to  $M_1$. In particular, isometry of  $M_0$ or $M_1$  induces an isometry of the whole manifold.

\subsection{The structure of eigenvalues of $A$} 

We continue with the proof of Proposition \ref{prop:main}. 
We assume that $g$ is a complete  simply connected  Riemannian metric on a connected $M^n$ of dimension $n\ge 2$; we assume  
$\dim(\operatorname{Sol})=2$  and the  existence of a projective transformation $\phi$ which is not an affine transformation  such that its action on $\operatorname{Sol}$ is as in the first case of \eqref{matrices} with $ c  \ne  \bar c$.    Our general goal is to prove that the sectional curvature of $g$ is positive constant. We can assume without loss of generality that $c>\bar c>0$ (since if both $c$ and $\bar c$ are negative, we consider $\phi^2$ instead of $\phi$).

 Our first goal (Proposition \ref{lemma3} below) is to show  that there exists a solution of the metrisability  equation such that the conditions (\ref{item1}-\ref{itemlast}) of \S \ref{intmul} are fulfilled.

First  consider, for each $k\in \mathbb{\mathbb{Z}}$, 
 the $(1,1)$-tensor $A_k$ constructed by the solutions $ \sigma + \bar \sigma$  and ${\phi^k}^*(\sigma + \bar \sigma)$ by  the formula 
    \eqref{eq:18}: we tensor  multiply  the inverse of the weighted tensor $\sigma + \bar \sigma$ by ${\phi^k}^*(\sigma + \bar \sigma)$ 
     and contract with respect of one upper and one low index. In the terms of projectively equivalent metrics $g$ and $\bar g(k) := {\phi^k}^*(g)$ this tensor is given by \eqref{L}.

 Next, take a point $p$ and  consider a basis in $T_pM$  such that 
 in this basis $\sigma$ and $\bar \sigma$ are given by diagonal matrices and $g$ is given by the identity matrix (we choose the way  (B) of representing weighted tensors by matrices).  
\begin{equation}\label{diagonal}
g= \diag(1,1,...), \   \sigma= \diag(s_1, s_2,...), \ \bar \sigma=  \diag(\bar s_1, \bar s_2,...).
\end{equation}

 Since ${\phi^k}^* \sigma= c^k \sigma$ and ${\phi^k}^* \bar \sigma= \bar c^k \bar \sigma$ we obtain that in this basis 
 the matrix of
  $A_k$ is  $$
 A_k= \diag\left( { 
 c^k s_1 + \bar c^k\bar s_1 ,  c^k s_2 + \bar c^k \bar s_2} ,...\right). 
 $$
 Since the metrics $g$ and therefore  ${\phi^k}^*g$ are positive definite, the eigenvalues of $A_k$ 
 must be  positive.  Since it should be true for any $k$,   each diagonal element  $s_1,s_2,...$ and also $\bar s_1, \bar s_2,...$
  is greater than or equal to zero.  Since $g$ corresponds to  $\sigma + \bar \sigma$, we actually 
 have  that each $s_i$  satisfies $0\le s_i \le  1$ and $\bar s_i= 1 - s_i$. 
   We also see that if $s_i=0$, then the corresponding eigenvalue of $A_k$ is  constant equal to $c^k$. 
    Similarly,  if  $s_i=1$  
 then 
 the corresponding eigenvalue  of $A_k$ is  $\bar c^m$

 Let us show now that each 
  $A_k$ has at most three eigenvalues and at most one   eigenvalue is different from $c^k$ and $\bar c^k$.
   Moreover, the multiplicities $m, \bar m$ of   the eigenvalues  $c^k$ and $\bar c^k$ satisfy $c^{m+1} =-\bar c^{ -\bar m -1}$.

  Consider  the basis in $T_xM$
   as above: the metric $g$ and the solutions $\sigma, \bar \sigma$ are given by \eqref{diagonal}. We will assume that the point  $p$ is generic.  Take $k=1$;  
  Assume that the eigenvalue  $c$ has multiplicity $m$  
   and  the eigenvalue  ${\bar c}$ 
  has multiplicity $2\bar m$. We allow of course that $m$ and/or  $\bar m $
   are zero.  
 Then, the number of $s_i$ different from $1$ and $0$ is  $(n-m- \bar m)$.

 \begin{proposition} \label{lemma3} If at least  one of the following two conditions, $$n- m- \bar m= 1 \textrm{  and } c^{m+1} =\bar c^{ -\bar m -1},$$ is not fulfilled, 
 then the  metric has constant sectional curvature. 
 \end{proposition}

   {\bf Proof. } 
   Let us consider the (1,1)-tensor 
   $G_k=g^{is}\bar g(k)_{sj}$, where $\bar g(k):= {\phi^k}^* g .$ Clearly, $ g$, ${\phi^k}^* g $ 
   and    $G_k$ are related by 
   ${\phi^k}^* g( \cdot ,\cdot )= g(G_k\cdot , \cdot)$ and $G_k$,  and $A_k$ are related by    \begin{equation} \label{Gt} G_k= \tfrac{1}{\sqrt{\det(A_k)}}A_k^{-1}.\end{equation} We take a generic point $p$ and consider a basic in $T_pM$ such that in this basis $g$,
     $\sigma$ and $\bar \sigma$ has the diagonal form \eqref{diagonal}; we assume now that the first $ n-  m  -  \bar m$ diagonal elements of $\sigma$ are not zero and denote them by $s_1, s_2,...$,
    the next $m$ elements 
   are  equal to $1$, and the remaining $\bar m$ elements are zero. Then, in this  basis, the matrix of $A_k$ is given by 
   $$
   \diag\left(\underbrace{s_1 c^k + (1-s_1) \bar c^k,s_2 c^k  + (1-s_2) \bar c^k,...}_{n -\bar m - m}, \underbrace{c^k,...}_{m}, \underbrace{\bar c^k,...}_{\bar m}\right).   
   $$
   Note that $n -\bar m - m\ge 1$, since otherwise the projective transformation is an affine transformation, as it follows from 
   Remark \ref{formula:sinjukov}.
   In view of \eqref{Gt},  the matrix of $G_k$  in this basis 
   is given by
   \begin{equation}\label{matrixGt}
c^{-km}\bar c^{- k\bar m}  \prod_{i=1}^{n-m-\bar m}\frac{1}{s_1 c^k + (1-s_1) \bar c^k} 
   \diag\left(\underbrace{\tfrac{1}{s_1 c^k + (1-s_1) \bar c^k },\tfrac{1}{s_2 c^k + (1-s_2) \bar c^k},...}_{ n - \bar m -   m}, \underbrace{c^{-k},...}_{  m}, \underbrace{\bar c^{-k},...}_{\bar m}\right).    
   \end{equation}
 The eigenvalues of $G_k$ (which we denote by $\nu_1,...,\nu_{n-m-\bar m}$, $\nu$, $\bar \nu$) have therefore the following asymptotic behavior for $k\to  +\infty $ 
 and for $k \to -\infty$ \ (all products below run  from $1$ to $n-m-\bar m$): 
 
 \begin{equation} \label{as}
 \begin{array}{c c c c}
 k\to + \infty & \nu_i(k)\sim \frac{c^{-(n-\bar m+1)k}    \bar c^{ - \bar mk} }{s_i\prod s_j} & \nu(k) \sim 
 \frac{c^{-(n-\bar m+1)k} \bar c^{ -\bar m  k} }{\prod s_j} & \bar \nu(k)\sim \frac{c^{-(n-\bar m)k} \bar c^{-(\bar m+1)k  } }{\prod s_j}\\ &&&\\
   k\to - \infty & \nu_i(k)\sim \frac{c^{- mk}   \bar c^{-(n-m+1) k } }{(1-s_i)\prod (1-s_j)} &   \nu(k) \sim 
 \frac{c^{-(m+1)  k}\bar c^{ - (n-m) k } }{\prod (1-s_j)} &  \bar \nu(k)\sim \frac{c^{- mk } \bar c^{-(n-m+1)  k} }{\prod (1- s_j)}
 \end{array}
\end{equation}
 
 Our next goal is to show that, unless the sectional curvature is constant,    we have  
 \begin{equation}\label{equality} 
 c^{n-\bar m} \le  \bar c^{-(\bar m+1)} \textrm{ \  and \ }  c^{m+1} \ge (\bar c)^{m-n}. 
 \end{equation}
 Before showing this, let us remark that the inequalities \eqref{equality} immediately imply the Proposition. 
 Indeed, dividing  the first  inequality by the second one, 
 $$
c^{n-m-\bar m-1}  \le (\bar c)^{n-m-\bar m-1}.  
 $$
 Since by assumptions $c>\bar c$, this implies $n-m-\bar m=1$ as we claim. Now, substituting $n-m-\bar m=1 $ in \eqref{equality}, we obtain 
 $$c^{m+1} \le  \bar c^{-(\bar m+1)} \textrm{ \  and \ }  c^{m+1} \ge \bar c^{-(m+1)}, $$
 implying   $c^{m+1} =\bar c^{ -\bar m -1}$ as we want.

 Assume the first inequality of \eqref{equality} is not fulfilled. The proof in the case when the second inequality of \eqref{equality} is not fulfilled is similar (and  actually, the first and the second inequalities  interchange  when we replace $\phi$  by $\phi^{-1}$ which of course corresponds to the automorphism 
  $k\longleftrightarrow -k$ of $\mathbb{Z}$).

Then, by \eqref{as},  all eigenvalues of $G_k$ decay exponentially for $t\to \infty$. Consider the sequence of the points $p, \phi^1(p), \phi^{2}(p),\phi^{3}(p),...$ . This sequence is a Cauchy sequence. 
Indeed, since  all eigenvalues of $G_k $ decay exponentially with $k\to \infty$,
   the distance  between $\phi^{\ell }(p)  $ and  $\phi^{(\ell+1)}(p)  $  also decays 
    at least exponentially for  $\ell \to \infty$ and the sequence is a Cauchy sequence.

 Since the manifold is complete, the Cauchy sequence $p, \phi^1(p), \phi^{2}(p),\phi^{3}(p),...$.  
 converges; we denote its limit by $P$.

  We    consider the  projectively invariant tensors we constructed in \S \ref{sec:4}:  Weyl tensor $W^i_{\ jk\ell}$ if $n\ge 3$ and Liouville tensor $L_{ijk}$ if $n=2$. Next, consider  
 the following smooth 
 function $F$ on our manifold: if $n\ge 3$, put 
 $$
 F  = W^i_{\ js\ell} W^{i'}_{\ j's'\ell'} g_{ii'}  g^{jj'}g^{s'}g^{\ell\ell'}.
 $$
 If $n=2$, put $F= g^{ii'}  g^{jj'}g^{ss'} L_{ijs}
L_{i'j's'}.$
  
 Since the function is continuous, we have $F(P)= \lim_{k \to \infty} F(\phi^k(p))$. From the other side, since $W$ and $L$
are  projectively invariant, \begin{equation} \label{pro}  F(\phi^k(p))=  
 W^i_{\ js\ell} W^{i'}_{\ j's'\ell'} \bar g(k)_{ii'}  \bar g(k)^{jj'}\bar g(k)^{ss'}\bar g(k)^{\ell\ell'} \end{equation}  
for dimension $n\ge 3$  and 
\begin{equation} \label{Pro1}  F(\phi^k(p))=  
 g(k)^{ii'}  g(k)^{jj'}g(k)^{s's} L_{ijs}
L_{i'j's'}.  \end{equation}

We consider a basis in  $T_pM$ such that the matrices of $g$, $\sigma$ and $\bar \sigma$ are as in  \eqref{diagonal}. 
We see that   in dimension 2 the sum \eqref{Pro1} is a the sum of nonnegative numbers    $\left(L_{ijs}\right)^2 $ with coefficients which are  products of reciprocals  to the diagonal entries of $G(k)$ and therefore grow exponentially for $k\to \infty$ in view of \eqref{as}. Thus, would  at least one of the numbers  $L_{ijs} $ be different from zero, 
the sum \eqref{Pro1} would be unbounded for $k\to \infty$. But it is bounded since it converges, for $k\to \infty$, to $F(P)$. Thus, $L_{ijs}$ is zero at the point $p$, and since the point $p$ was generic, $ L_{ijs}$ is identically zero and hence, by  Proposition \ref{prop:dim2}, the metric has constant sectional curvature. Now, by \cite[Corollary 6]{Matveev2007} (or \cite{Bonahon1993}) the sectional curvature is positive and we are done.

If the dimension $n\ge 3$, essentially the same idea works but one should be slightly more accurate, and the reason for it that in the formula \eqref{pro} we multiply 3 times by the  reciprocals of some  diagonal components of $G(k)$, and one times by a  diagonal component of $G(k)$. Since different components of $G(k)$ have different asymptotic, one may conceive the situation  when the sum \eqref{pro}  is bounded. Let us explain how we overcome this difficulty.

First consider  an example. Suppose  the index  $i$ lies in 
 $\{n - \bar m-  m+1  ,...,n-m\}$,  the  index  $j$ lies  in  $ \{ n - \bar m-  m+1  ,...,n-m\}$,    and the indices  $s$ and $\ell $ lie in $\{ n - \bar m+1  ,...,n\}$. Then,  the sum \eqref{pro} is a sum of nonnegative terms  containing $\left(W^{i}_{\ js\ell}\right)^2$  
  multiplied by a positive coefficient which behaves, for $k\to \infty$ and up to multiplcation with a positive constant, as 
 $$\underbrace{c^{-(n-\bar m+1)k}\bar c^{ - \bar m k}}_{\nu(k)}\underbrace{  c^{ (n-\bar m+1)k} \bar c^{ \bar m  k} }_{1/\nu(k)}
 \underbrace{ c^{(n-\bar m)k} \bar c^{(\bar m+1)  k}}_{1/\bar \nu(k)}  \underbrace{ c^{(n-\bar m)k}  \bar c^{ (\bar m+1) k}}_{1/\bar \nu(k)}=  c^{2(n-\bar m)k} \bar c^{2(\bar m+1) k}.$$
Indeed, the coefficient  $g_{ii'}$ in \eqref{pro} is (up to a positive constant)
$c^{-(n-\bar m+1)k}\bar c^{ - \bar m k}$, the coefficient $g_{jj'}$ is also $ c^{-(n-\bar m+1)k}\bar c^{ - \bar m k}$ so  it cancels with $g_{ii'}$, 
   and   the coefficients $g^{ss'}$ and $g^{\ell\ell'}$ are, up  to a  positive multiple, $c^{(n-\bar m)k}  \bar c^{ (\bar m+1) k}$.

Let us now show that $W^i_{\ js\ell}=0$, if  at least one of the indices $j,s, \ell$ lies in  $    \{1,..., n-  \bar m- m\}$. Assume that this is not the case. Then, as in the example above, one shows that  the sum \eqref{pro} contains the term 
$$
  \left(W^{i}_{\ js\ell}\right)^2. 
$$
multiplied  by a coefficient that 
  behaves for $k\to \infty$ at least as  $ c^{(n-\bar m)k}\bar c^{(\bar m+1)k }$.  This give us a contradiction unless $W^{i}_{\ js \ell}=0$.

Thus, for any vector $v^i$ such that all components of $v^i$  with $i>n-m-\bar m$   are zero, we have   $W^i_{\ jk\ell} v^j=0$. Then, the Weyl tensor has a (nontrivial) nullity  at the point $p$  in the terminology of \cite{GoverMatveev2015}. Since the point $p$ was generic,    Weyl tensor  has a  nullity  everywhere. Metrics with this condition   were studied in  \cite{GoverMatveev2015}, in particular  it was shown there, see \cite[Theorem 37]{GoverMatveev2015},  the existence of an symmetric tensor $\Phi_{ij}$ such that it is  projectively invariant   and such that it vanishes if and only if the metric $g$ is an Einstein metric. Note that the tensor $\Phi_{ij}$ may fail to be smooth, but it is always continuous, which is sufficient for our goals. Now, repeating with the tensor $\Phi_{ij}$ the same arguments we did with $L_{ij s}$, we obtain that it must vanish (all indexes are low so the problem we had with  $W^{i}_{\ js \ell}$ and which was due to  the upper index   does not appear). Thus, $\Phi_{ab}\equiv 0$, so the metric is Einstein by \cite[Corollary 3.17]{GoverMatveev2015}. 

Now,  complete 
 Einstein metrics (of dimension $\ge 3$) do not allow  nonaffine projective transformations by  \cite{KiosakMatveev2009}, unless the sectional curvature is constant and positive.  Proposition \ref{lemma3} is proved.

 \begin{remark} 
We have seen that the proof of  Proposition \ref{lemma3} contains two important steps: in the first step we have shown that (if what we claim is not fulfilled) then for each generic point $p$ 
 the sequence $\phi(p)$,  $\phi^2(p)$,... converges. In the second step we analyzed projectively invariant  tensors constructed in  \S \ref{sec:4}  and have (in the more complicated case of dimension $\ge 3$) that the Weyl tensor has nullity; then certain nontrivial results of \cite{GoverMatveev2015} and \cite{KiosakMatveev2009}. 
The rough  scheme of the proof of the remaining case will be essentially the same, but the arguments will be more delicate. In particular in order to show convergence we need to improve our projective transformation by composing it with a certain isometry, and in the proof that Weyl tensor has a nullity is based on additional observations.   
 \end{remark}

 \subsection{Remaining step in the proof of Proposition \ref{prop:main} }
 
 Thus, the only remaining case is when $n-m-\bar m=1$ (which means that besides the constant eigenvalues $0$ and $1$ we have only one eigenvalue which we denote by $\lambda$.)  Precisely this situation we considered in \S  \ref{intmul}. Moreover, we have that $c^{m+1} =\bar c^{ -\bar m -1}$. Then, the asymptotic behavior  \eqref{as} reads as follows: 
 \begin{equation} \label{as1}
 \begin{array}{c c c c}
 k\to + \infty & \nu_1(k)\sim  \tfrac{1}{s_1^2} \left(\tfrac{\bar c}{c}\right)^k    & \nu(k) \sim 
 \tfrac{1}{s_1 } \left(\tfrac{\bar c}{c}\right)^k & \bar \nu(k)\sim \tfrac{1}{s_1 }  \\ &&&\\
   k\to - \infty & \nu_1(k)\sim  \tfrac{1}{(1-s_1)^2} \left(\tfrac{ c}{\bar c}\right)^k  &   \nu(k) \sim 
 \tfrac{1}{(1-s_1)}   &  \bar \nu(k)\sim \tfrac{1}{(1-s_1)} \left(\tfrac{  c}{\bar c}\right)^k 
 \end{array}
\end{equation}

 Suppose first $M_0$ or $M_1$  defined in \S  \ref{intmul} is not empty. W.l.o.g., we can assume that  $M_0$ is not empty. 
 Clearly, the projective transformation $\phi$ preserves the sets $M_0$. Since $M_0$ contains at most two connected components, w.l.o.g. we can think that $\phi$ preserves each connected component  of $M_0$.     The asymptotic above  induces  contraction to  $M_0$ for $k\to +\infty$ and for $k\to -\infty$, so the restriction of the projective transformation 
  to $M_0$ is an isometry of $M_0$ w.r.t. the induced metric. But as we explained in \S \ref{isometry}, an isometry of $M_0$ induces an isometry of $M$ which we call $\psi$; the superposition $\psi^{-1} \circ \phi$ is also  a projective transformation which is not an affine transformation and the matrix $T$ for it coincides with $T_\phi$. Thus, we may  w.l.o.g.  replace $\phi$ by  $\psi^{-1} \circ \phi$, which implies that we assume  that each point of $M_0$ is   a fixed point of $\phi$.

  Take any generic point $p$ in $M$ and consider the shortest geodesic connecting $p$ with $M_0$, the endpoing of the geodesic lying at $M_0$ will be denote by $P$.  The geodesic is orthogonal to $M_0$ and therefore its velocity vector is at each point an eigenvector of $A$ with eigenvalue $\lambda$. 
 Then,  by the asymptotic \eqref{as1} above, the sequence $p, \phi(p), \phi^2(p),...,$ converges; clearly, its limit is  the point  $P$. 
 
 Let us now assume that the dimension is $n\ge 3$ and 
 consider  the function $F$ given by \eqref{pro}. 
 Arguing as in the end of  the proof of Proposition \ref{lemma3},  using continuity of this function, 
 we obtain that the components   $W^i_{\ 1k\ell}$ at the point $p$ 
  are zero for $i\ne 1$. But for $i=1$ it is also zero, since the component $R^1_{\ 1jk}$ of the curvature tensor is zero because of the symmetries of the curvature tensor, and the  component  $\delta^1_{\ k}  R_{1j}   - \delta^1_{\ j}  R_{1k}$  vanishes because the vector $v_1$ is an eigenvector of the Ricci tensor by \cite[Lemma 1]{KiosakMatveev2009}. Finally, the Weyl tensor has nullity. 
  
  Now, in the case both $m$ and $\bar m$ are not zero, the function $B$ from \cite{GoverMatveev2015} corresponding to the nullity is constant by  \cite[\S 5]{GoverMatveev2015}, and  the metric $g$  has  constant sectional curvature by \cite{KiosakMatveev2010}. 
  
  Thus, the remaining case is when $\bar m=n-1$ and $m=0$. In this case,  $M_0$ is a point or two points, and the metric  has a  concircular vector fields by   \cite[\S 5]{GoverMatveev2015}, which vanishes at the points of $M_0$. Then, the  isometry group of the manifold  contains $SO(n)$ which has fixed point at the points of $M_0$ and whose induced action on the tangents space to the points of $M_0$ is the standard action of $SO(n)$. The generic 
  orbits of   this  action are integral manifolds of the  distributions of the eigenspace of the eigenvalue $1$; therefore, they are compact (and in fact they are diffeomorphic to the $n-1$ spheres).

  Consider now   the projectively invariant tensor  $\Phi_{ij}$ from  \cite[Theorem 37]{GoverMatveev2015}. Arguing as at the end of the proof of Proposition \ref{prop:main}, we obtain that the velocity vectors of geodesic passing through points of $M_0$ lie in  the kernel of $\Phi_{ij}$.
  
  Now, the restriction of the projectively invariant tensor  $\Phi_{ij}$ to these integral manifold is either zero or nondegenerate. The second case is impossible since applying projective transformations $\phi$, $\phi^2$,$\phi^3$,...  to an orbit of the action   we obtain as the limit a  point of $M_0$ and the integral of $\sqrt{\operatorname{det}(\Phi)}$  over it is zero. Thus, $\Phi_{ij}$
    is  identically  zero, so the  metric $g$ is an Einstein metric and we are done by \cite{KiosakMatveev2009}.
  
  Similar, but more simple arguments work in dimension $n=2$: instead of formula  \eqref{pro} we need to use \eqref{Pro1}; arguing as above, we obtain that only the component 
  $L_{222}$  may be different from zero. But the component  $L_{222}$ must be zero because of the symmetries of $L$, and we are done. 
  
 Finally, the only remaining case is when $M_0=M_1=\emptyset$. We show that this case is impossible. Indeed, as explained in \cite[\S 4]{Matveev2003} (and follows directly from the splitting-gluing procedure for projectively equivalent metrics obtained in  \cite{BolsinovMatveev2011}, see also  \cite{BolsinovMatveev2015}),  in this case   
 the manifold is the direct product of $\mathbb{R}\times N  \times \bar N$, where $N$ is a simply-connected  $m$-dimensional manifold equipped with the metric $h$, $\bar N$ is a simply-connected  $\bar m$-dimensional manifold equipped with the metric $  \bar h$, and the metric on the manifold is given by the Levi-Civita formula \eqref{LCformula}, where $x_1$ is the coordinate on $\mathbb{R}$, 
 $x_2,...,x_{m +1}$ are local coordinates on $N$, and  $x_{m+2},...,x_{n}$ are local coordinates on $\bar N$. Since $M$ is complete, both manifolds $(N,h)$ and $(\bar N, \bar h)$ are complete.

Any isometry of $N$ or of $\bar N$ induces an isometry of $M$; and  any projective transformation of $M$ induces a homothety of $N$ and of $\bar N$. In the case if the projective transformation $\phi$ induces an isometry $\psi$  of $N$ or $\bar N$, one can ``correct'' $\phi$ with the help of $\psi$ (as we did above) such that the induces action of $\phi$ on $N$ or on $\bar N$ is identical.  If   the projective transformation $\phi$ induces a nonisometric homothety   of $N$ (or $\bar N$), then $N$ (or $\bar N$) is isometric to the euclidean $\mathbb{R}^n$ which allows a transitive group of isometries. In all cases, for any point $p \in M$, 
  by correcting $\phi$ by an isometry of $N$ and of $\bar N$  we can achieve that the geodesic passing through $p$ such that its initial velocity vector is an eigenvector of $A$ corresponding to $\lambda$ is invariant with respect to the projective transformation.   
  Using asymptotic \eqref{as1}, we imply then that the sequence $p, \phi(p), \phi^2(p), \phi^3(p),...$ converges. At the limit point we clearly have that    $\lambda$ is then equal to $0$ or to $1$, which implies that $M_0$ or $M_1$ are not empty and contradicts our assumptions.  
This finishes the proof of Proposion   \ref{prop:main}, and therefore the proof of Theorem \ref{thm:new}.

\end{document}